\documentclass[12pt]{article}
\usepackage[utf8]{inputenc}
\usepackage[T1]{fontenc}
\usepackage{epsfig}
\usepackage{amsfonts}
\usepackage{amssymb}
\usepackage{amsmath}
\usepackage{amsthm}
\usepackage{mathtools}
\usepackage{latexsym}
\usepackage{graphicx}
\usepackage{bm}
\usepackage{tabularx}
\usepackage{booktabs} 
\usepackage{enumerate}
\usepackage{caption}
\usepackage[titletoc]{appendix}
\usepackage[numbers]{natbib}
\usepackage{bbm}
\usepackage{url}
\usepackage{subfig}
\usepackage{hyperref}
\usepackage{color}
\usepackage[all]{xy}
\usepackage{float}
\usepackage{mwe}
\usepackage{mathtools}
\usepackage{media9}
\usepackage{multirow}
\usepackage{blkarray}

\catcode`\@=11

\newcommand{\shortdot}[1]{\raisebox{-0.4pt}{$\stackrel{\bullet}{#1}$}}
\DeclareBoldMathCommand\boldlangle{\left\langle} 
\DeclareBoldMathCommand\boldrangle{\right\rangle}

\theoremstyle{plain}
\newtheorem{theorem}{Theorem}[section]
\newtheorem{lemma}[theorem]{Lemma}

\newtheorem{proposition}[theorem]{Proposition}

\newtheorem{algorithm}[theorem]{Algorithm}
\theoremstyle{definition}

\theoremstyle{remark}

\begin{document}

\title{A Stochastic Model for Electric Scooter Systems}
%\title{Charging Electric Scooter Systems: An Empirical Process Perspective}
\author{ 
    Jamol Pender \\ School of Operations Research and Information Engineering \\ Cornell University
\\ 228 Rhodes Hall, Ithaca, NY 14853 \\  jjp274@cornell.edu  \\ 
\and
Shuang Tao \\ School of Operations Research and Information Engineering \\ Cornell University
\\ 293 Rhodes Hall, Ithaca, NY 14853 \\  st754@cornell.edu  \\ 
\and
Anders Wikum \\ School of Operations Research and Information Engineering \\ Cornell University
\\ 206 Rhodes Hall, Ithaca, NY 14853 \\  aew236@cornell.edu  \\ 
 }

\maketitle
\begin{abstract}
Electric scooters are becoming immensely popular across the world as a means of reliable transportation around many cities.  As these e-scooters rely on batteries, it is important to understand how many of these e-scooters have enough battery life to transport riders and when these e-scooters might require a battery replacement.  To this end, we develop the first stochastic model to capture the battery life dynamics of e-scooters of a large scooter network.  In our model, we assume that e-scooter batteries are removable and replaced by agents called \textbf{swappers}.  Thus, to gain some insight about the large scale dynamics of the system, we prove a mean field limit theorem and a functional central limit theorem for the fraction of e-scooters that lie in a particular interval of battery life.  Exploiting the mean field limit and the functional central limit theorems, we develop an algorithm for determining the number of \textbf{swappers} that are needed to guarantee levels of probabilistic performance of the system.  Finally, we show through a stochastic simulation and real data that our stochastic model captures the relevant dynamics.   
\end{abstract}

%**************************************************************************
%**************************************************************************

\section{Introduction}

It's a bird. It's a plane. Nah, it's a scooter! Electric scooter (e-scooter) companies are growing in popularity across the United States looking to take advantage of the ride-sharing and micro-mobility economy by providing an alternative to cars and bicycles. E-scooters were first introduced in September of 2017 in Santa Monica \citet{Hall2018} when the micro-mobility company Bird Rides Inc. placed thousands of scooters all around the city.  Bird's scooters were immediately popular with commuters, since they are convenient and are a low cost alternative to cars. These e-scooters can reach speeds of up to 25 miles per hour., but actually speeds vary by municipality.  The e-scooter technology has become a prevalent form of transportation to provide a feasible solution to the last mile problem in the transportation literature. 

These web-based e-scooters are controlled by rental networks and are easily operated by smartphones. Customers are able to use e-scooters by downloading mobile applications to their smartphones. The mobile applications (apps) then show customers an image of the nearest e-scooters to their GPS location and this information will direct customers to the nearest available e-scooter. A picture of the mobile application is given in Figure \ref{Bird_App}.  Using the information about available e-scooters, a customer will take an e-scooter and after completing their ride, customers can leave their e-scooter anywhere outside restricted zones, as indicated in the mobile application on their phone.

\begin{figure}[ht]
\captionsetup{justification=centering}
\begin{center}
		\includegraphics[scale = .3]{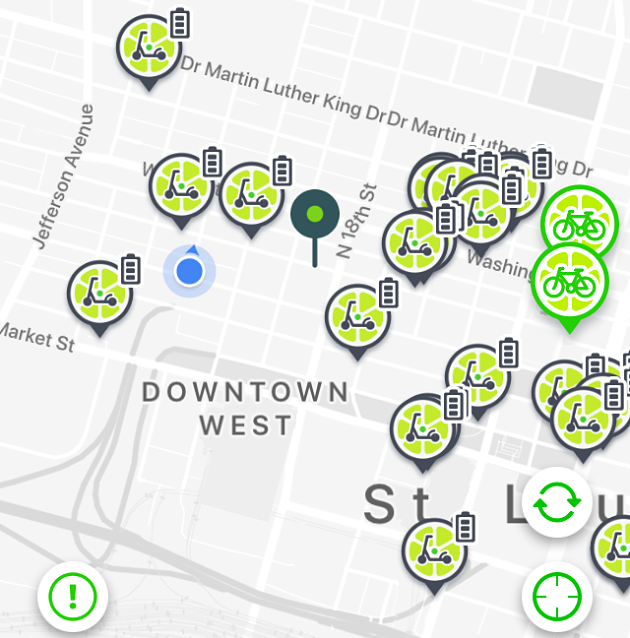}
		\end{center}
		\vspace{-.02in}
\caption{Picture of the Lime E-Scooter App. } \label{Lime_App}
\end{figure}

\begin{figure}[ht]
\captionsetup{justification=centering}
\begin{center}
		   \includegraphics[scale = .6]{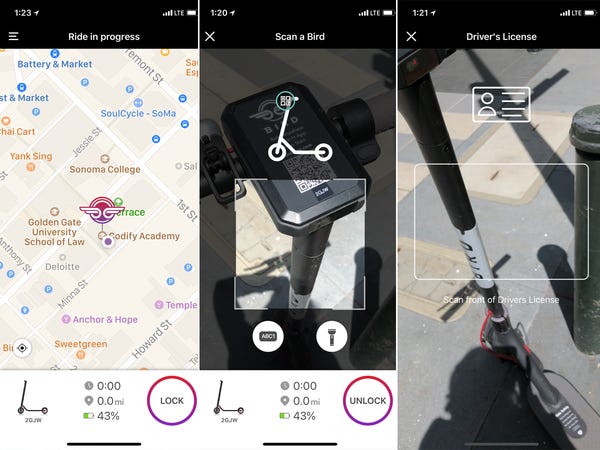}
		\end{center}
		\vspace{-.2in}
\caption{Picture of Bird E-Scooter App With Battery Life Remaining. } \label{Bird_App}
\end{figure}

Companies operating e-scooters are rapidly expanding operations in the United States.  For example, another company called Lime has a fleet of e-scooters that are available in more than 60 US cities.  They also have an international presence in over 10 cities as well.   Today, several major companies, including Bird and Lime, offer dockless e-scooter services, and several other companies, including the ride-sharing companies Uber and Lyft, have recently entered the market as the demand continues to grow.  Recent financial analyses show that Lime is valued over \$1 billion and its rival Bird is valued at more than \$2 billion.  This growth will continue to accelerate as the demand grows for these e-scooters in many of the largest cities around the world.  
  
While e-scooter transportation should reduce emissions, and automobile congestion in local areas, e-scooters are not without accidents, see for example \citet{allem2019electric, kobayashi2019merging, carville2018}.  Although they are a convenient and affordable solution to transportation gaps, they are operated by batteries and need to be charged.  The battery charging operations for these e-scooters consists of private individuals who go around and collect scooters to charge them.  Bird scooter collectors are called “Bird Hunters” and Lime scooter collectors are called "Lime Juicers".  These collectors can make significant profits if they charge these scooters, see for example \citet{goshtasb2018proposing}.  Generally, there is a flat payment for charging the e-scooter and the payment will increase depending on the difficulty of locating the scooter.  Typically the "Bird Hunters" will take the scooters home and after charging them overnight, need to drop them off as groups of three at dedicated assigned points called “Nests” \citet{bordes2019impacts}.  However, this particular way of charging e-scooters is not without incident.  There have been several situations where juicers and hunters quarrel over the ability to charge the e-scooters \citet{goshtasb2018proposing}.  There are also several places where juicers and hunters "own" a specific territory to charge the e-scooters.  It would of great interest to eliminate this territorial behavior over charging the e-scooters.    

To solve some of these issues, several e-scooter production companies are developing e-scooters with easily removable batteries.  In this situation, instead of having "juicers" and "hunters" that need to charge the batteries, batteries are swapped by employees of the scooter company.  We call these agents who replaced the dead batteries "swappers".  Many large scale e-scooter systems such as Bird and Lime already give their users the ability to see real-time availability and "battery life" through a smartphone app and web API.  A picture of the Bird app showing that $43\%$ battery life for an e-scooter is given in Figure \ref{Bird_App}.   This information helps users make better decisions such as where to pick up e-scooters with enough battery life to get to their destination. It is also helpful for "swappers" who will replace the batteries when the scooter's battery life is below a predetermined threshold.  

However, in order to implement the "swapping" for a scooter company, we need to understand the battery life dynamics of a large scale e-scooter system.  Understanding these dynamics will enable us to determine how many of these swappers are necessary to achieve the ideal performance of the e-scooter system.  To this end, in this paper, we develop the first stochastic model that analyzes this "swapping" process along with the battery dynamics of e-scooters.  Our goal in this work is to understand the dynamics of a removable e-scooter system and what proportion of the e-scooters have a particular fraction of battery life.  In particular, we focus on understanding how many scooters are actually available to customers at a specific time when they want to travel.  Our analysis yields new insights for staffing removable battery e-scooter systems that will be used in the future.

%**************************************************************************
%**************************************************************************

 \subsection{Main Contributions of Paper}

In this section, we describe the contributions of our work in this paper. 
\begin{itemize}
\item We construct the first stochastic e-scooter model using empirical processes, which measure the battery life dynamics. Since our model is difficult to analyze for a large number of scooters, we propose to analyze an empirical process that describes the proportion of scooters that have a certain fraction of battery life remaining.  Our model is informed by real data collected from the JUMP API in Washington D.C.
\item We prove a mean field limit and a central limit theorem for our stochastic e-scooter empirical process, showing that the mean field limit and the variance of the empirical process can be described by a system of $\frac{K^2+3K}{2}$ differential equations where $K$ is the number equally sized intervals of battery life.    
\item We develop a novel algorithm based on our limit theorems for staffing the number of swappers that are needed to ensure that the proportion of scooters that have a small amount of battery is smaller than a given threshold.  
\end{itemize} 

%**************************************************************************
%**************************************************************************

\subsection{Organization of Paper}

The remainder of this paper is organized as follows. Section~\ref{Data} describes how we obtained and analyzed our e-scooter data.  It also includes insights on how the data was used to inform our model and model parameters.  In Section \ref{Sec_Bike_Model}, we introduce two stochastic e-scooter models.  The first model assume that battery usage is instantaneous , while the second model assumes that battery usage occurs according to an exponential distribution.  In Section, \ref{Mean_Field_Limit}, we prove the mean field limit of our stochastic e-scooter model showing that the mean field limit is a system of $K$ differential equations.  In Section \ref{Central_Limit}, we prove a functional central limit theorem for the empirical process. We also show the variance of the diffusion limit can be approximated by a system of ordinary differential equations that are coupled to the mean field limit.  In Section \ref{Staffing}, we use the mean field and central limit theorems to construct a staffing policy for the number of "swappers" need to satisfy probabilistic performance constraints.  In Section \ref{Numerics}, we show that our results are indeed valid by comparing them to a stochastic simulation of the e-scooter system.  In Section \ref{Conclusion}, we conclude and give directions for future work.  Finally, additional proofs and theorems for our second model are given in the Appendix or Section \ref{Appendix}.  

\subsection{Preliminaries of Weak Convergence}

Following \citet{ko2018strong}, we assume that all random variables in this paper are defined on a common probability space $(\Omega, \mathcal{F}, \mathbb{P})$.  Moreover, for all positive integers $k$, we let $\mathcal{D}([0 , \infty), \mathbb{R}^k)$ be the space of right continuous functions with left limits (RCLL) in $\mathbb{R}^k$ that have a time domain in $[0, \infty)$.  As is usual, we endow the space $\mathcal{D}([0 , \infty), \mathbb{R}^k)$ with the usual Skorokhod $J_1$ topology, and let $M^k$ be defined as the Borel $\sigma$-algebra associated with the $J_1$ topology. We also assume that all stochastic processes are measurable functions from our common probability space $(\Omega, \mathcal{F}, \mathbb{P})$ into $(\mathcal{D}([0 , \infty), \mathbb{R}^k), M^k)$.   Thus, if $\{\zeta\}^\infty_{n=1}$ is a sequence of stochastic processes, then the notation $\zeta^n \rightarrow \zeta$ implies that the probability measures that are induced by the $\zeta^n$'s on the space $(\mathcal{D}([0 , \infty), \mathbb{R}^k), M^k)$ converge weakly to the probability measure on the space $(\mathcal{D}([0 , \infty), \mathbb{R}^k), M^k)$ induced by $\zeta$.  For any $x \in (\mathcal{D}([0 , \infty), \mathbb{R}^k), M^k)$  and any $T > 0$, we define 
\begin{equation}
||x||_T \equiv \sup_{0 \leq t \leq T}  \ \max_{i = 1,2,...,k} |x_i(t)|
\end{equation}
and note that $\zeta^n$ converges almost surely to a continuous limit process $\zeta$ in the $J_1$ topology if and only if 
\begin{equation}
||\zeta^n - \zeta||_T \to 0 \quad a.s.
\end{equation}
for every $T > 0$.

\section{Insights from Electric Scooter Data} \label{Data}

In this section, we describe some of the e-scooter trip data that we collected for this paper. This data is used to inform our stochastic models in the subsequent sections.  With the real data, we can understand how e-scooter battery levels change when riders take trips, the rate of arrival to use e-scooters, and how long riders use the e-scooters in terms of time duration and distance. Below we describe how we collected raw geographic bike-share data, reconstructed likely trips, and filtered some data points that did not make sense from a rider perspective. 

\subsection{The Data Collection Process}
General Bikeshare Feed Specification (GBFS) is an industry standard for sharing bike-share data that has been adopted by virtually every bike-sharing company, thanks in no small part to it being a requirement for operation in many US cities. GBFS is designed to provide a real-time snapshot of a city’s fleet, which includes vehicle locations and battery levels for bikes that are not in active use.  This information is also collected without keeping records of trips or personal information. In addition to providing more detailed monthly reports to the District Department of Transportation, maintaining a public API with GBFS data is a condition of operating in the Washington D.C. metro area. 

\begin{figure}[ht]
\captionsetup{justification=centering}
\begin{center}
		\includegraphics[scale=.65]{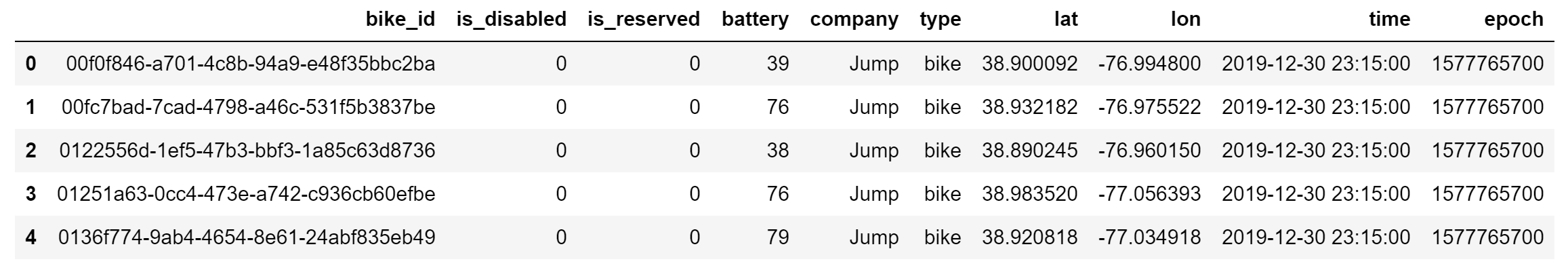}
		\end{center}
		\vspace{-.2in}
\caption{Raw GBFS Data} \label{raw_data}
\end{figure}

The dataset analyzed in this paper is based on GBFS data scraped from APIs maintained by JUMP for the D.C. metro area. The data was pulled from the JUMP API once per minute from the dates 01/01/2020 to 03/01/2020 during peak hours (6:00 – 23:59) to coincide with vehicle location updates. After adding time stamps, the scraped data has the form shown in Figure \ref{raw_data}.  Note that the data has a vehicle type and some of them say "bike".  This is because JUMP operates both e-bikes and e-scooters and the data is collected for both.  Since our analysis is centered around e-scooters, we removed the data for the e-bikes.  

\begin{table}[H]
\caption{GBFS Data Fields} % title of Table
\centering % used for centering table
\begin{tabular}{c c } % centered columns (4 columns)
\hline\hline %inserts double horizontal lines

\hline % inserts single horizontal line
$bike\_id$ & Unique identification String for each vehicle in fleet.\\
$is\_disabled$ & Boolean, 1 if vehicle is outside of approved geo-fences and 0 otherwise. \\
$is\_reserved$ & Boolean, 1 if vehicle is reserved through the JUMP app, 0 otherwise.\\
$battery$ & Integer between 0 and 100 representing percent battery remaining.\\
$company$ & Company that operates the vehicle.\\
$type$ & Vehicle type, one of `bike' or `scooter'.\\
$lat/lon$ & GPS location of vehicle at the given timestamp.\\
$time$ & Date/Time when API call is executed.\\
$epoch$ & Time API call is executed, in seconds since 01/01/1970.\\
\hline %inserts single line
\end{tabular}
\label{table:nonlin} % is used to refer this table in the text
\end{table}

\paragraph{Reconstructing E-Scooter Trips from Raw Data}
We are ultimately interested in trip data for the purpose of informing the model parameters for our stochastic models. Though GBFS data explicitly excludes trip records, we were able to reconstruct trips by observing the times and locations at which bikes disappeared from and reentered the GBFS dataset.  Using this information, our goal was to determine if a trip actually occurred or was it something else like rebalancing or strange movements of the e-scooter.  

More explicitly, we compute the haversine distance between the start and end GPS coordinates.  The start location is defined to be where the e-scooter first disappears from the data extraction from the API and the end location is where the e-scooter reappears in the data again.  To mitigate noise from potential rebalancing, any disappearance with a distance less than 50 meters was removed from the dataset.  Each of the remaining disappearances were tentatively designated as trips, with corresponding start and end times, locations, and battery levels. The distance of the trip was determined by the haversine distance between its start and end GPS coordinates, which is likely an underestimate of the true distance traveled on the e-scooter. Thus, when measuring the e-scooter drain rates, we are definitely overestimating this quantity in our analyses.  

\begin{figure}[ht]
\captionsetup{justification=centering}
\begin{center}
		\includegraphics[scale=.6]{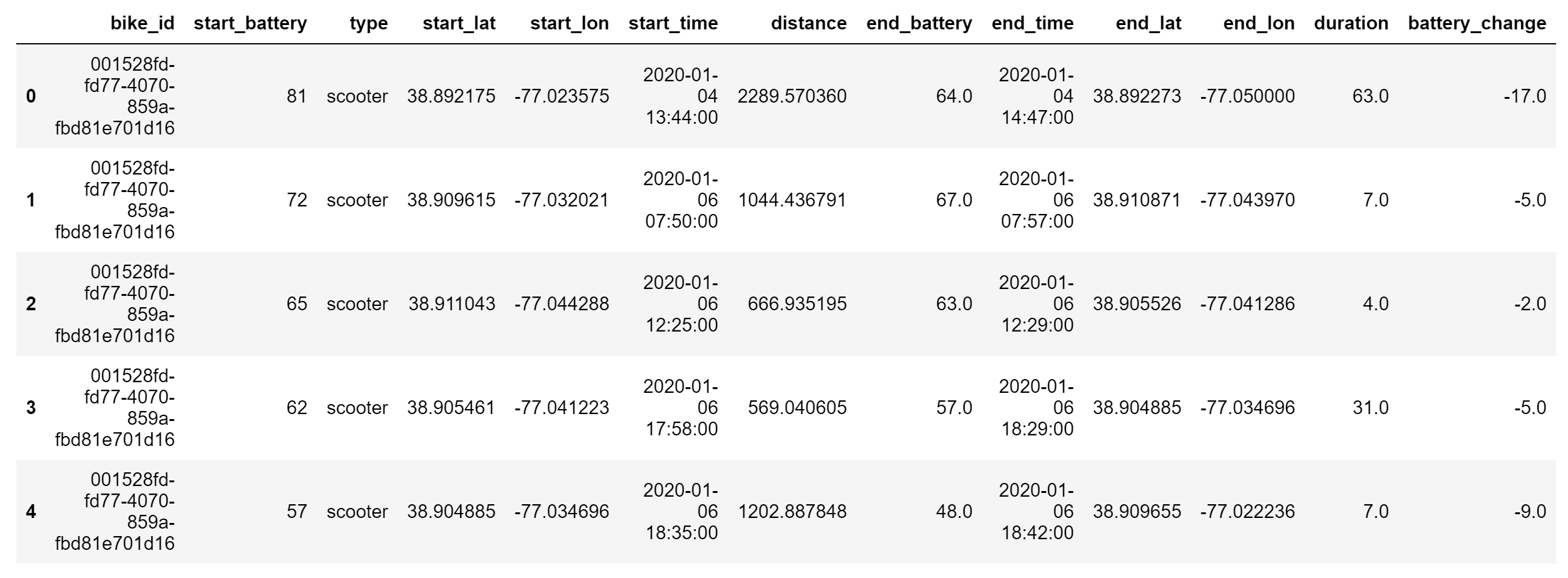}
		\end{center}
		\vspace{-.2in}
\caption{Cleaned Trip Data} \label{trip_data}
\end{figure}

Our final data filtering steps included removing trips corresponding to overnight disappearances, trips with an average velocity greater than the maximum theoretical speed of JUMP e-bikes (25 mph), and trips corresponding to recharging in which the e-scooter battery life increased. After the filtering process was complete, we were left with a dataset of 71,518 likely trips with format shown in Figure \ref{trip_data} .

\subsection{Insights Gained from Data}

% \begin{figure}[ht]
%\captionsetup{justification=centering}
%\begin{center}
%		\includegraphics[scale=.5]{./histogram_cdf_distance.png}
%		\end{center}
%		\vspace{-.2in}
%\caption{Scooter Rental Distance Distribution.  } \label{Scooter_Distance}
%\end{figure}

\paragraph{Distance Traveled} In this section, we use the data we collected to understand how long riders travel on the e-scooters.  On the right of Figure \ref{Scooter_Duration_cdf}, we plot a histogram of the distances traveled by riders on the top plot. Since most of the data  has less than 5 kilometers distance, we restricted the dataset to be less 5 kilometers.  On the bottom right of Figure \ref{Scooter_Duration_cdf}, we plot the cumulative density function (cdf) of the distance data.  This cdf plot allows us to understand the quantiles or percentiles of the data more clearly than the histogram plot.  In this context, we observe that the median distance or 50\% quantile is roughly equal to 800 meters or close to half a mile.  Moreover, we observe that about 80\% of riders are traveling less than 1.5 kilometers.

\paragraph{Time Duration of Riders}
On the left of Figure \ref{Scooter_Duration_cdf}, we plot a histogram of the time spent by riders with the e-scooters on the top plot.  Since most of the data was less than one hour, we restricted the dataset to be less than one hour or 60 minutes.  On the bottom left of Figure \ref{Scooter_Duration_cdf}, we plot the cumulative density function (cdf) of the time duration data.  This cdf plot allows us to understand the quantiles or percentiles of the data more clearly than the histogram plot.  In this context, we observe that the median duration is roughly 8 minutes and about 80\% of riders are traveling less than 15 minutes on an e-scooter. 

 \begin{figure}[ht]
\captionsetup{justification=centering}
\begin{center}
		\includegraphics[scale=.3]{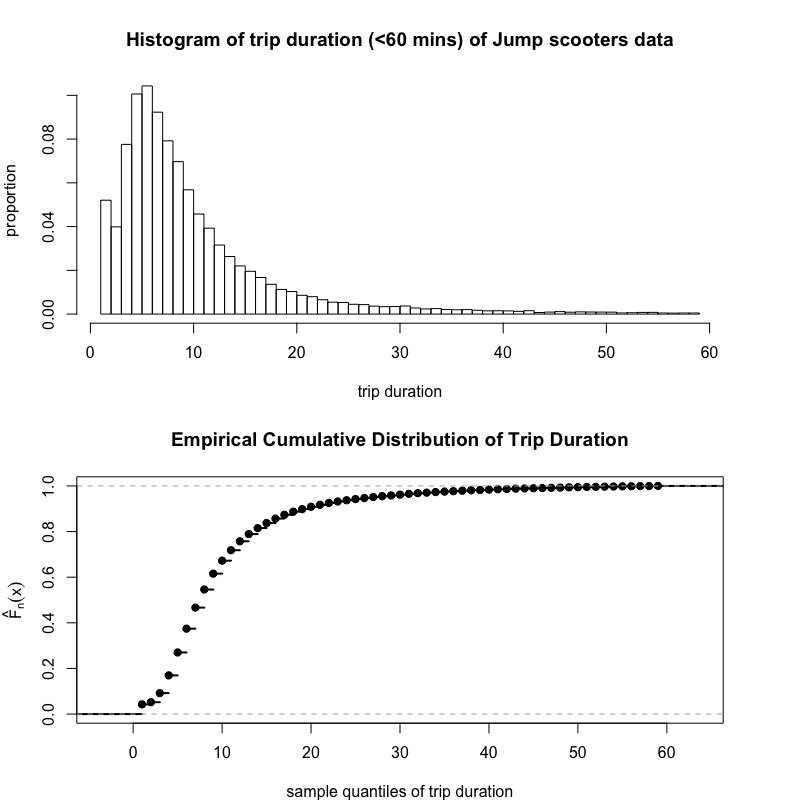}~\includegraphics[scale=.3]{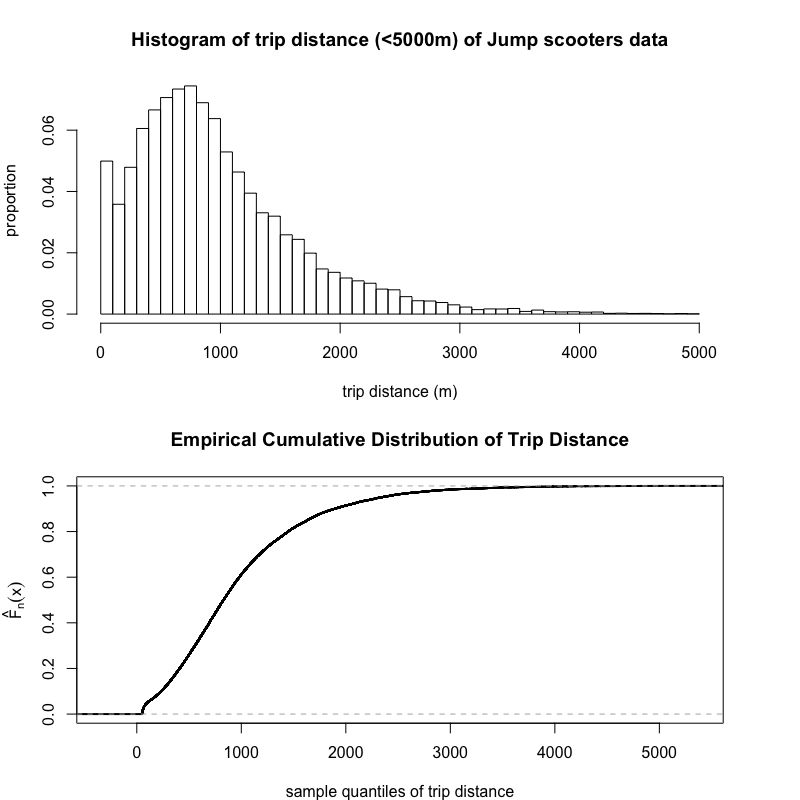}
		\end{center}
		\vspace{-.2in}
\caption{Scooter Rental Duration Distribution (Left).  Scooter Rental Distance Distribution (Right). } \label{Scooter_Duration_cdf}
\end{figure}

% \begin{figure}[ht]
%\captionsetup{justification=centering}
%\begin{center}
%		\includegraphics[scale=.5]{./histogram_cdf_duration.png}
%		\end{center}
%		\vspace{-.2in}
%\caption{Scooter Rental Duration Distribution.  } \label{Scooter_Duration_cdf}
%\end{figure}

\paragraph{Inter-arrival Times of Riders}
On the left of Figure \ref{Scooter_interarrival_cdf}, we plot a histogram of the inter-arrival times of riders to e-scooters in the network.  This histogram provides information about how many riders we should expect to arrive to the system during a time period.  We should mention that the arrival rate should depend on time, however, we ignore this time dependence when looking at the inter-arrival times here.  Since most of the data for the inter-arrival times was less than 20 minutes, we restricted the dataset to be less than 20 minutes.  On the bottom left of Figure \ref{Scooter_interarrival_cdf}, we plot the cumulative density function (cdf) of the inter-arrival data.  We observe that the median duration is roughly equal to 1 minute, however, this information is a bit misleading because of how the data is collected.  Since the scooter API is updated only once per minute, it is impossible to observe an inter-arrival time less than one minute.  Moreover, we observe that about 95\% of the inter-arrival times are less than three minutes in length.   On the right of Figure \ref{Scooter_interarrival_cdf}, we plot the arrival rate as a function of time average over the days of the data set.  It is clear that the arrival rate is non-stationary and varies over the time of day.  The two hump pattern (one in the morning and one in the afternoon) is also observed in this data and is common in ride-sharing data.  

\begin{figure}[ht]
\captionsetup{justification=centering}
\begin{center}
		\includegraphics[scale=.3]{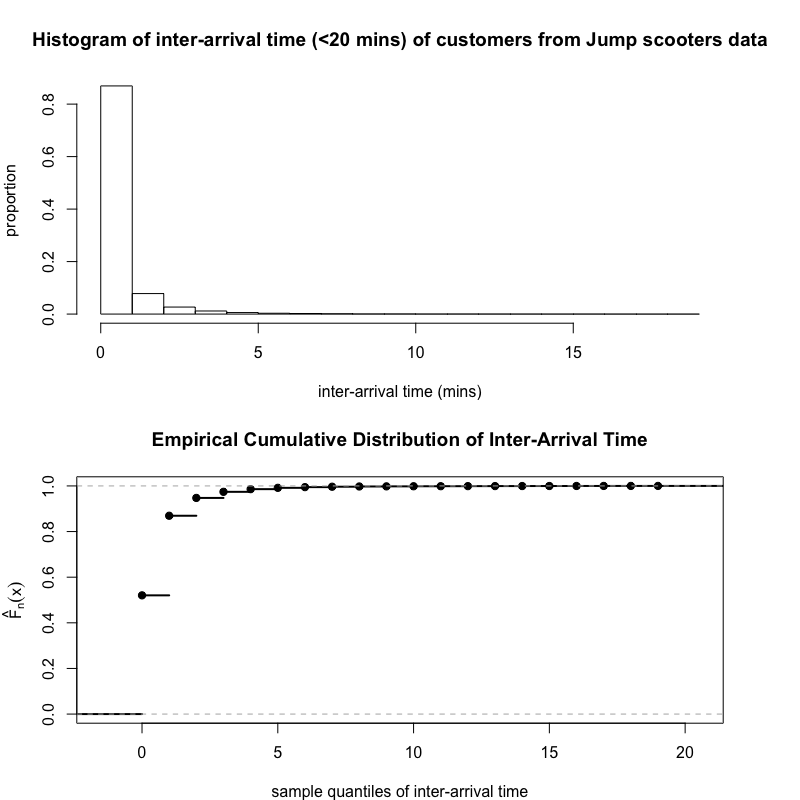}~\includegraphics[scale=.3]{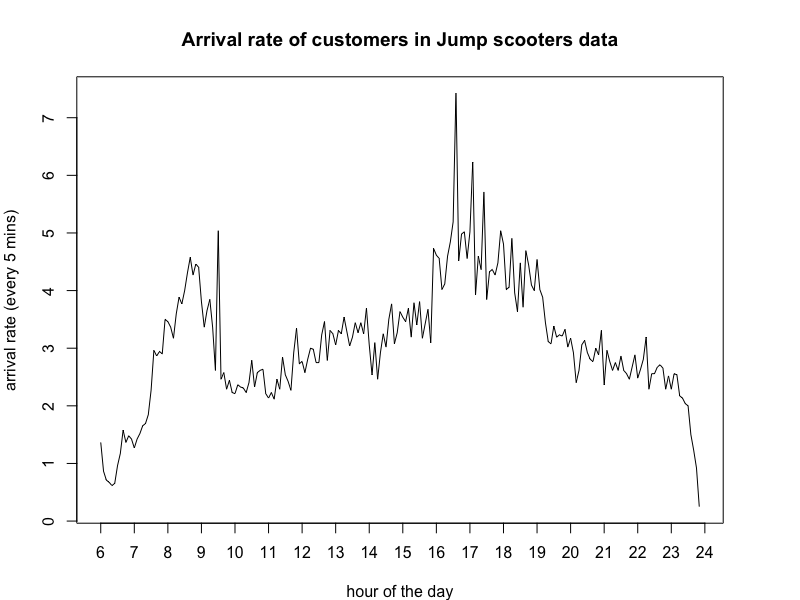}
		\end{center}
		\vspace{-.2in}
\caption{Scooter Rental Inter-Arrival Time Distribution (Left). Scooter Rental Arrival Rate Throughout the Day (Right).  } \label{Scooter_interarrival_cdf}
\end{figure}

%
% \begin{figure}[ht]
%\captionsetup{justification=centering}
%\begin{center}
%		\includegraphics[scale=.4]{./arrival_rate.png}
%		\end{center}
%		\vspace{-.2in}
%\caption{Scooter Rental Arrival Rate Throughout the Day.  } \label{Scooter_arrivalrate}
%\end{figure}

%\begin{figure}[ht]
%\captionsetup{justification=centering}
%\begin{center}
%		\includegraphics[scale=.5]{./reg_distance.png}
%		\end{center}
%		\vspace{-.2in}
%\caption{Scooter Trip Distance vs. Change in Battery Life.} \label{Scooter_Distance_regression}
%\end{figure}
%
%\begin{figure}[ht]
%\captionsetup{justification=centering}
%\begin{center}
%		\includegraphics[scale=.5]{./reg_duration.png}
%		\end{center}
%		\vspace{-.2in}
%\caption{Scooter Trip Duration vs. Change in Battery Life.} \label{Scooter_Duration_regression}
%\end{figure}

\begin{figure}[ht]
\captionsetup{justification=centering}
\begin{center}
		\includegraphics[scale=.3]{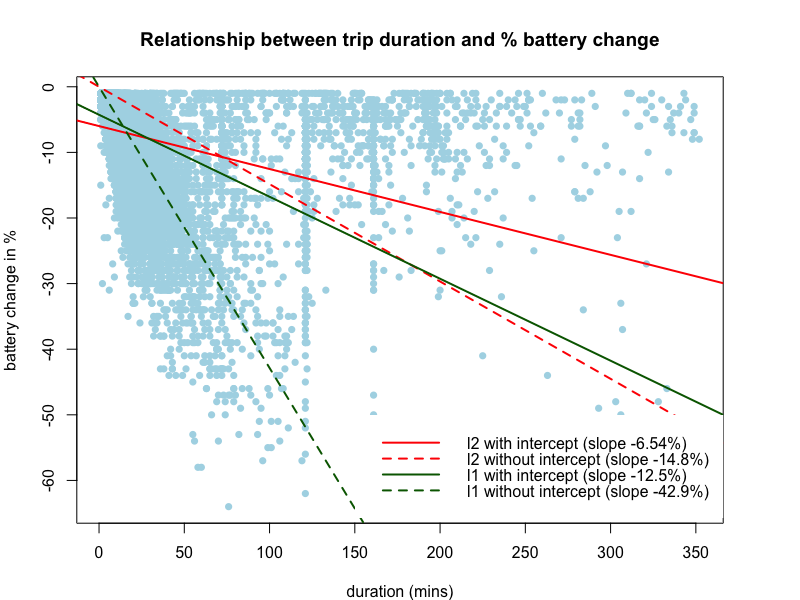}~\includegraphics[scale=.3]{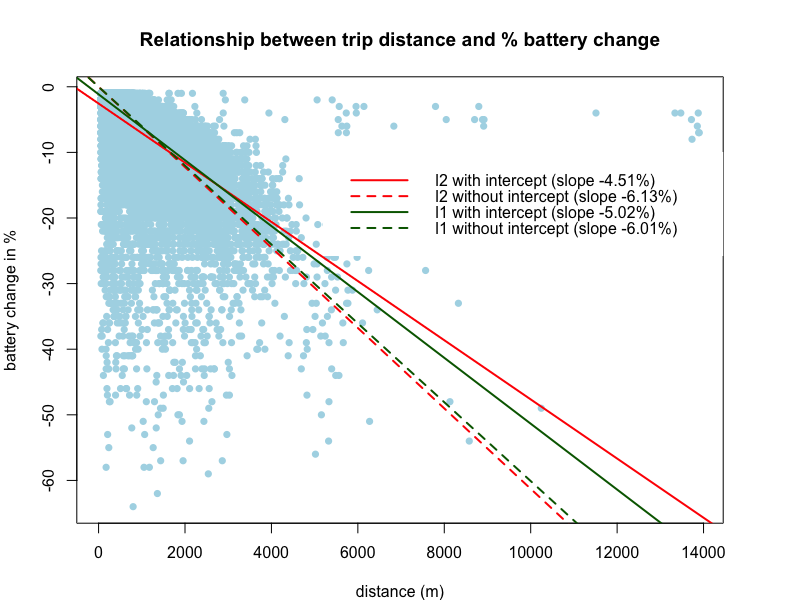}
		\end{center}
		\vspace{-.2in}
\caption{Scooter Trip Duration vs. Change in Battery Life (Left). Scooter Trip Distance vs. Change in Battery Life (Right). } \label{Scooter_Duration_regression}
\end{figure}

\paragraph{Estimating Battery Life of E-Scooters}  Another informative measurement from a data perspective is the battery usage dynamics of riders.  On the left of Figure \ref{Scooter_Duration_regression}, we show a scatterplot of trip time duration and the decrease of battery life.  The plot measures for each trip, how long the customer used the e-scooter and what was the subsequent drain in the battery life of the e-scooter.  On the right of Figure \ref{Scooter_Duration_regression}, we show a scatterplot of trip distance and the decrease of battery life.  This plot measures for each trip, how far in meters the customer drove the e-scooter and what was the subsequent drain in the battery life. We should emphasize that the e-scooter was not tracked during the entire time of usage and only the starting and ending GPS locations were used to compute the haversine distance between them.  In both plots of Figure \ref{Scooter_Duration_regression}, we observe that the relationship between time or distance with battery life is negative. We use regression analysis to explore these relationships in a formal way.  Table \ref{table:reg_result} summarizes the coefficients of different regression methods used to understand the relationship between \% battery change and distance/duration.

\begin{table}[h]
\begin{center}
\caption{Regression coefficients for different methods} \label{table:reg_result}% title of Table
\begin{tabular}{c|c|c|c} % centered columns (4 columns)
\hline\hline %inserts double horizontal lines
feature & norm & intercept & slope \\
\hline
\multirow{4}{*}{distance}& \multicolumn{1}{l}{$L_2$} & \multicolumn{1}{l}{0} & \multicolumn{1}{l}{-6.13\%}\\
                         &\multicolumn{1}{l}{$L_2$} & \multicolumn{1}{l}{-2.55} & \multicolumn{1}{l}{-4.51\%}\\\cline{2-4}
                         & \multicolumn{1}{l}{$L_1$} & \multicolumn{1}{l}{0} & \multicolumn{1}{l}{-6.01\%}\\
                         &\multicolumn{1}{l}{$L_1$} & \multicolumn{1}{l}{-1.15} & \multicolumn{1}{l}{-5.02\%} \\\hline
\multirow{4}{*}{duration}& \multicolumn{1}{l}{$L_2$} & \multicolumn{1}{l}{0} & \multicolumn{1}{l}{-14.8\%}\\
                         &\multicolumn{1}{l}{$L_2$} & \multicolumn{1}{l}{-5.98} & \multicolumn{1}{l}{-6.54\%}\\\cline{2-4}
                         & \multicolumn{1}{l}{$L_1$} & \multicolumn{1}{l}{0} & \multicolumn{1}{l}{-42.9\%}\\
                         &\multicolumn{1}{l}{$L_1$} & \multicolumn{1}{l}{-4.25} & \multicolumn{1}{l}{-12.5\%} \\\hline                   
\hline %inserts single line
\end{tabular}
\end{center}
\end{table}

We observe that in both plots of Figure \ref{Scooter_Duration_regression} that performing the regression without an intercept increases the absolute value of the negative slopes.  Moreover, we see that $L_1$ regression generally yields more negative slopes than their $L_2$ counterparts.  This is especially true in the duration vs. battery plot on the left of Figure \ref{Scooter_Duration_regression}.  Finally, we observe that the trip distance seems to be a better estimate of real trips versus the time duration of trips.  This is consistent with the estimates of battery life that are reported from the Jump scooter company.  

\begin{figure}[ht]
\captionsetup{justification=centering}
\begin{center}
		\includegraphics[scale=.3]{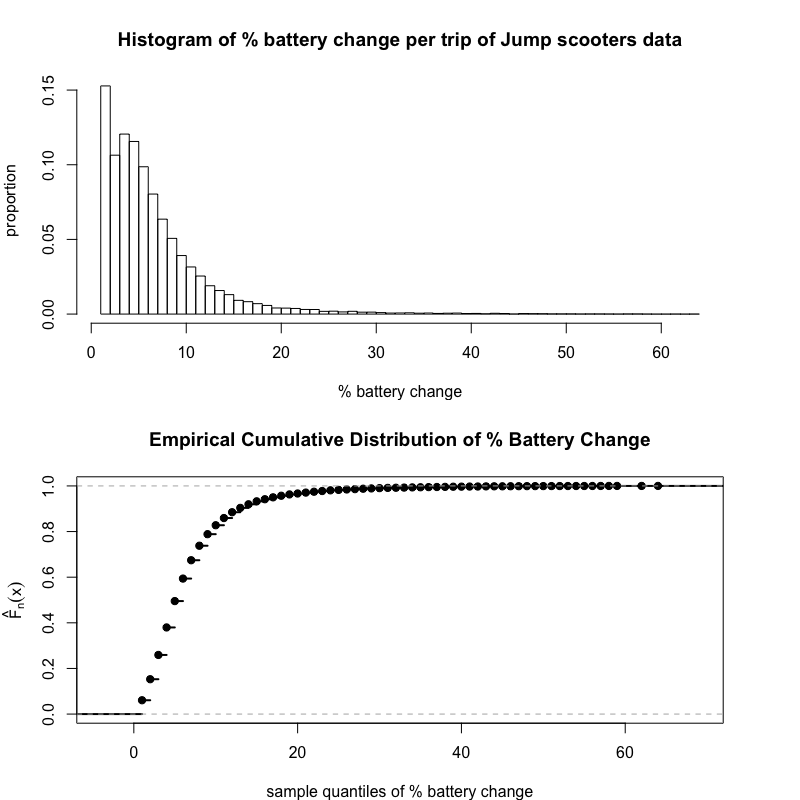}
		\end{center}
		\vspace{-.2in}
\caption{Scooter Rental \% Battery Change Distribution.  } \label{Scooter_Battery_cdf}
\end{figure}

\paragraph{Battery Usage Per Trip}
Finally, on the top of Figure \ref{Scooter_Battery_cdf}, we plot a histogram of the battery used by customer.  Since most of the data was less than 60\% of the total battery life, we restricted the dataset to be less than 60 \%.  We observe that the histogram for battery usage almost looks exponentially distributed, but probably is more of a gamma distribution.  On the bottom of Figure \ref{Scooter_Battery_cdf}, we plot the cdf of the battery usage data.   We observe that the median  battery usage is about 6\% of battery life  and about 80\% of riders are using less than 10\% of battery life of the scooter on each trip. From the data, we also observe that the minimum starting battery is about 21\%. Thus, the average scooter can do about 13 trips before its battery needs to be swapped out for a new one.

%**************************************************************************
%**************************************************************************

\section{Stochastic Models and Limit Theorems} \label{Sec_Bike_Model}

In this section, we propose two new Markovian queueing models for the empirical process of e-scooters battery life.  In our first model, we assume the battery usage time is instantaneous.  However, in our second model, we assume the battery usage time is an exponentially distributed random variable.  Although one can analyze each e-scooter individually, this is a high dimensional stochastic process and given the large scale of these e-scooter systems, we wish not to take on this intractable endeavor.   As a result, we resort to using an empirical process perspective for the e-scooter system.  The empirical process perspective reduces the dimension from the number of e-scooters, which is large, to the number of intervals of battery life one would like to keep track of. For example, in Washington D.C, each scooter company is allowed to operate at most 2500 scooters.  However, the number of intervals of battery life for the scooters is at most 100.  This represents an 25 fold reduction in dimension.   However, it may not be necessary to even keep that much granularity for the purposes of this work.  We suggest a value of $K = 10$ to give battery intervals of $10\%$.    This would yield a 250 fold reduction in dimension.  

There is a large literature in the space of bike sharing and the sharing economy, see for example  \citet{hampshire2012analysis, nair2013large, schuijbroek2017inventory, faghih2017empirical, singla2015incentivizing, jian2016simulation, freund2020data}.  Despite there being much research on bike sharing networks there is much less literature on electric scooters and their impact on transportation networks in large cities.  Our goal in this work is to add to the growing literature in the sharing economy, but specifically for e-scooters.  Our approach leverages new data resources for the scooters and uses the data to inform the structure of the stochastic models we will build in the sequel.  Our new stochastic models leverage techniques from empirical process theory and weak convergence of martingales.     
 
Empirical processes are not new and have been used in a variety of contexts in queueing theory, see for example \citet{graham1997stochastic, graham2000kinetic, graham2000chaoticity, graham2005functional, li2014mean, li2016mean, mitzenmacher2016analyzing, ying2016approximation, iyer2011mean, iyer2014mean, yang2016mean, yang2018mean, yang2019information}.  One of the first papers to consider empirical processes in ride-sharing is \citet{mohamed2012mean}, where the authors model bike sharing networks as a network of finite capacity single server queues.  Using empirical processes, they prove a mean field limit theorem for the number stations that have $k$ bikes.  In this context, the dimensionality is reduced from the total number of bikes in the network to roughly the size of the largest station.  In large metropolitan cities like New York City and Washington D.C., this reduction is huge and useful.    Recently, \citet{tao2017stochastic}, prove a central limit theorem for the same bike sharing model, showing that the central limit theorem is quite good at describing the fluctuations of the stochastic bike-sharing network process.  Moreover, recent work by \citet{fricker2016incentives, bortolussi2016mean, li2016bike, li2016queueing, li2017fluid, el2018using} has also generalized the mean field limit theorems of \citet{mohamed2012mean} to the setting of non-stationary bike sharing systems and for Markovian arrival processes (MAPs) for the arrival and service distributions.  More recently, \citet{graef2019fractional} extend the mean field model from ordinary differential equations to fractional ordinary differential equations.  Generally, the mean field limit theorems provide rigorous support for using ordinary differential equations for describing the mean dynamics of the empirical measure. However, \citet{graef2019fractional} shows that using fractional ordinary differential equations might be more appropriate as they provide more flexibility than their non-fractional counterparts. Before we get specific about the models that we will describe in the sequel, we give some of the common notation that we will use throughout the remainder of the paper below in Table \ref{table:nonlin}. 

\begin{table}[H]
\caption{Summary of Notation} % title of Table
\centering % used for centering table
\begin{tabular}{c c } % centered columns (4 columns)
\hline\hline %inserts double horizontal lines

\hline % inserts single horizontal line
$i$ & index for e-scooter\\
$N$ & Number of e-scooters \\ % inserting body of the table
$N^*$ & Number of swappers\\
$K$ & Battery life bucket size\\
%% $K_R$ & Battery threshold for recharging\\
$K_U$ & Battery threshold for riding\\
$\lambda$ & Arrival rate of recharger to a e-scooter\\
$\mu$ & Arrival rate of customer to a e-scooter\\
$p_{ij}$ & Probability of battery end up in bucket $j/K$ in a ride with starting battery in bucket $i/K$\\
$B_i(t)$ & Battery life of e-scooter $i$ at time $t$\\
$Y^N(t)$ & Empirical process of e-scooters battery life at time $t$\\
\hline %inserts single line
\end{tabular}
\label{table:nonlin} % is used to refer this table in the text
\end{table}

\subsection{Model 1: Instantaneous Battery Usage}\label{model_1}
Here we describe our first model for modeling the battery dynamics of an e-scooter network.  We consider an empirical process of the battery life among all e-scooters in the system. The goal here is to model the distribution of battery life as Markov process and study the asymptotic behavior of the system as the number of e-scooters grows towards infinity i.e, $N \to \infty$.  More specifically, we analyze a mean field and central limit theorem for the empirical process to help understand how different parameters can affect the system's performance.

\subsubsection{Modeling Assumptions}
%Table \ref{table:nonlin} summarizes the notation that we use throughout the remainder of the paper.

\paragraph{Customer Arrival (battery usage):}
We assume that customers arrive to the system following a Poisson process with rate $\mu N$ (uniform on geographical location). Only e-scooters with battery life above a certain threshold $K_U/K$ can be picked up and used by the customer. For simplicity of the model, we assume that after the customer picks up the e-scooter, the battery life changes immediately according to a probability matrix $P=(p_{ij})_{ij}$.  Each element in the matrix  $P=(p_{ij})_{ij}$ represents an e-scooter moving from the $i^{th}$ interval of battery life to the $j^{th}$ interval.  Using the Jump scooter data we collected, we find the following empirical probability matrix $\hat{P}$ when setting $K=5$, which is equal to

$$\hat{P}=\begin{blockarray}{cccccc}
[0,20\%] & [20\%,40\%] & [40\%,60\%] & [60\%,80\%] & [80\%,100\%] \\
\begin{block}{(ccccc)c}
0 & 0 & 0 & 0 & 0 & [0,20\%] \\
0.097 & 0.903 & 0 & 0 & 0 & [20\%,40\%]  \\
0.003&	0.345&	0.652 & 0 &0& [40\%,60\%] \\
0.0008 &	0.022 & 0.329 &	0.648 & 0 & [60\%,80\%]\\
0.00 & 0.004 &	 0.021 & 0.446 & 0.529 & [80\%,100\%] \\
\end{block}
\end{blockarray}$$
Here the first row of $\hat{P}$ is zero because the minimum starting battery life is around 21\% from the Jump scooters data.

\paragraph{Swapper Arrival:}
We assume that swappers arrive to the system following a Poisson process with rate $\lambda N^*$. The probability of a e-scooter with battery life in bucket $k/K$ getting recharged is based on a choice model $\frac{Y_k^N(t)g_k}{\sum_{i=0}^{K-1}Y_i^N(t)g_i}$, where $\{g_i\}_{i=0}^{K-1}>0$ is a decreasing sequence on $i$. For simplicity of the model, we assume that after the swappers picks up the e-scooters, the battery life jumps to full immediately (i.e. neglecting swapping time).

\paragraph{Remark:} Note that instead of only recharge e-scooters with low battery, here we use a choice model for recharging that gives more weight to e-scooters with low battery life.  With the choice model, there is a positive probability to recharge a e-scooter in bucket $\left[\frac{K-1}{K},1\right]$.  Without the choice model, we cannot guarantee the Lipschitz property of the drift function for the limiting mean field equations.  The Lipschitz property is crucial for proving the mean field and central limit results in this work. However, one can set up the choice model $\{g_i\}_{i=0}^{K-1}$ so that the probability of recharging a e-scooter with high battery is low (In fact we only need $\min_{i}\{g_i\} > 0$).

\subsubsection{Markov Jump Process}
Now that we have described the dynamics of the model, we are now free to construct our empirical process model.  To this end, we define the empirical process $Y_k^N(t)$ as the proportion of e-scooters with remaining battery life between $[\frac{k}{K},\frac{k+1}{K})$.  Thus, we can write $Y_k^N(t)$ as the following equation
$$Y_k^N(t)=\frac{1}{N}\sum_{i=1}^{N}\mathbf{1}\left\{\frac{k}{K}\leq B_i(t) < \frac{k+1}{K}\right\}, \quad k=0,\cdots, K-1$$
where $N$ is the total number of e-scooters in the system, and $B_i(t)$ is the battery life of the $i^{th}$ e-scooter at time $t$.  Moreover, we also assume that battery drainage of a single ride follows a discrete probability distribution, i.e.
$$P(\text{battery end up in } j/K \text{ but started with battery in } i/K)=p_{ij}, \quad j\leq i=0,\cdots, K-1.$$  If we condition on $Y_k^{N}(t)=y_k$, the transition rates of $y$ are specified as follows.

\paragraph{Swapping Batteries:}
When there is a swapper arriving to the system to swap an e-scooter's battery with battery life in the interval $[\frac{k}{K},\frac{k+1}{K}]$, the proportion of e-scooters with battery life in bucket $[\frac{k}{K},\frac{k+1}{K}]$ goes down by $1/N$, the proportion of e-scooters with battery life in bucket $[\frac{K-1}{K},1]$ goes up by $1/N$, and the transition rate $Q^N$ is 
\begin{eqnarray}
Q^{N}\left(y,y+\frac{1}{N}(\mathbf{1}_{K-1}-\mathbf{1}_{k}) \right) &=&  \lambda N^* \frac{y_kg_k}{\sum_{i=0}^{K-1}y_ig_i} .
\end{eqnarray}

\paragraph{Riding a Scooter:}
When there is a customer riding an e-scooter with battery life in the interval $[\frac{k}{K},\frac{k+1}{K}]$ where $k\geq K_U$, the proportion of e-scooters with battery life in the interval$[\frac{k}{K},\frac{k+1}{K}]$ moves down by $1/N$ with probability $p_{kj}$.  In addition, the proportion of e-scooters with battery life in the interval $[\frac{j}{K},\frac{j+1}{K}]$ moves up by $1/N$, and the transition rate $Q^N$ is 
\begin{eqnarray}
Q^{N} \left(y,y+\frac{1}{N}(\mathbf{1}_{j}-\mathbf{1}_{k}) \right) &=&
\mu N p_{kj} y_k \mathbf{1}\{k\geq K_U\} \mathbf{1}\{j\leq k\}.
\end{eqnarray}

With the described transitions, one can show that $Y^N(t)$ is a Markovian jump process with the above transition rates.  In Figure \ref{model1_diagram}, we illustrate the transitions between states in the model proposed above.

\begin{figure}[H]
\captionsetup{justification=centering}
\begin{center}
		\includegraphics[scale = 0.6]{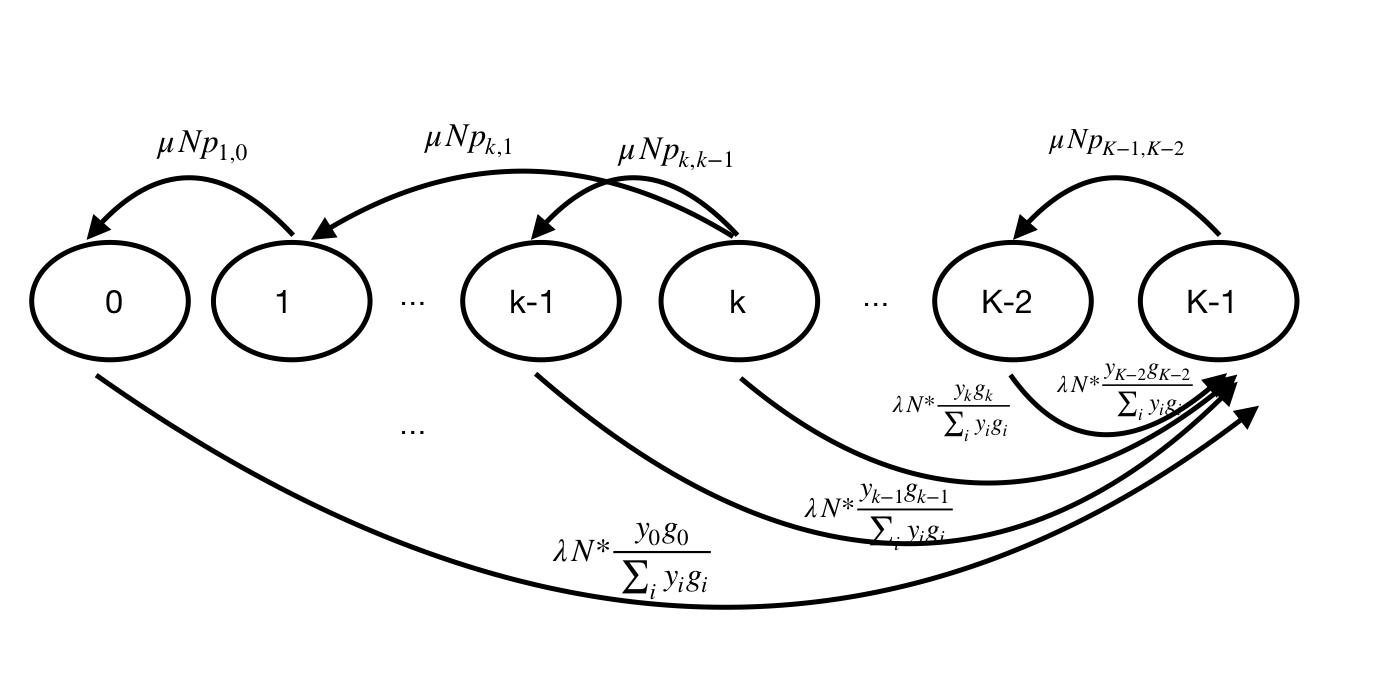}
		\end{center}
		\vspace{-.02in}
\caption{Diagram of transitions of states for the empirical process $Y^N(t)$.} \label{model1_diagram}
\end{figure}

Despite some complexity of our model, it is not completely realistic as we make the battery usage transitions instantaneous.  What follows in the sequel is a generalization of our first model where the battery life transitions have an exponential distribution.  

\subsection{Model 2: Exponentially Distributed Battery Usage Time}\label{model_2}	

In this subsection, we propose a different model for the empirical process of battery life where battery usage time is considered to be exponentially distributed with rate $\mu_U$. We use the same notation from Section \ref{model_1} and introduce a new variable $R^N(t)$ as the number of e-scooters in use (riding by customers) at time $t$. Table \ref{table:nonlin2} summarizes the additional notation we need for this new model.

\begin{table}[H]
\caption{Summary of Additional Model 2 Notation} % title of Table
\centering % used for centering table
\begin{tabular}{c c } % centered columns (4 columns)
\hline\hline %inserts double horizontal lines

\hline % inserts single horizontal line
$1/\mu_U$ & Mean trip duration\\
$R^N(t)$ & Number of e-scooters in use at time $t$\\
$X^N(t)$ & Proportion of e-scooters in use at time $t$\\
\hline %inserts single line
\end{tabular}
\label{table:nonlin2} % is used to refer this table in the text
\end{table}

\subsubsection{Modeling Assumptions}

Now we describe the following modeling assumptions we make for this new model.  

\paragraph{Customer Arrival (battery usage):}
We assume that customers arrive to the system following a Poisson process with rate $\mu (N-R^N(t))$ (uniform on geographical location). Only e-scooters with battery life above a certain threshold $K_U/K$ can be picked up and used by the customer. After the customer picks up the e-scooter, they will ride the e-scooter for a time that is exponentially distributed with rate $\mu_U$, and after this the battery life changes immediately according to a probability matrix $P=(p_{ij})_{ij}$.

\paragraph{Swapper Arrival:}
We assume that swappers arrive to the system following a Poisson process with rate $\lambda N^*\left(1-\frac{R^N(t)}{N}\right)$, which is proportional to the number of e-scooters available at the time. The probability of a e-scooter with battery life in bucket $k/K$ getting swapped is based on a choice model $\frac{Y_k^N(t)g_k}{\sum_{i=0}^{K-1}Y_i^N(t)g_i}$, where $\{g_i\}_{i=0}^{K-1}>0$ is a decreasing sequence on $i$. For simplicity of the model, we assume that after the swapper picks up the e-scooter, the battery life jumps to full immediately (i.e. neglecting the swapping time).

\subsubsection{Markov Jump Process}
Now that we have described the dynamics of the model, we are now free to construct our empirical process model.  To this end, we define the fraction of e-scooters in use as 
$$X^N(t)=\frac{R^N(t)}{N}$$

By conditioning on $(X^N(t),Y^{N}(t))=(x,y)$, the transition rates of $(x,y)$ are specified as follows:
\paragraph{Battery Swapping:}
When there is a swapper arriving to the system to swap the battery of an e-scooter with battery life in the interval $[\frac{k}{K},\frac{k+1}{K}]$, the proportion of e-scooters with battery life in the interval $[\frac{k}{K},\frac{k+1}{K}]$ goes down by $1/N$, the proportion of e-scooters with battery life in the interval$[\frac{K-1}{K},1]$ goes up by $1/N$, and the transition rate $Q^N$ is 
\begin{eqnarray}
Q^{N}\left((x,y),\left(x,y+\frac{1}{N}(\mathbf{1}_{K-1}-\mathbf{1}_{k})\right) \right) &=&  \lambda N^* (1-x)\frac{y_kg_k}{\sum_{i=0}^{K-1}y_ig_i}
\end{eqnarray}

\paragraph{Customer Arrival:}
When there is a customer arriving to the system to pick up a e-scooter with battery life in the interval $[\frac{k}{K},\frac{k+1}{K}]$ where $k\geq K_U$,, the proportion of e-scooters in use goes up by $1/N$, and the transition rate $Q^N$ is
\begin{eqnarray}
Q^{N}\left((x,y),\left(x+\frac{1}{N},y\right) \right) &=&  \mu N(1-x)\sum_{k=K_U}^{K-1}y_k
\end{eqnarray}

\paragraph{Battery Usage:}
When there is a customer riding a e-scooter with battery life in the interval $[\frac{k}{K},\frac{k+1}{K}]$ where $k\geq K_U$, the proportion of e-scooters with battery life in the interval $[\frac{k}{K},\frac{k+1}{K}]$ goes down by $1/N$, with probability $p_{kj}$ the  proportion of e-scooters with battery life in the interval $[\frac{j}{K},\frac{j+1}{K}]$ goes up by $1/N$, and the proportion of e-scooters in use goes down by $1/N$, and the transition rate $Q^N$ is 
\begin{eqnarray}
Q^{N} \left((x,y),\left(x-\frac{1}{N},y+\frac{1}{N}(\mathbf{1}_{j}-\mathbf{1}_{k})\right) \right) &=&
\mu_U Nx p_{kj} y_k \mathbf{1}\{k\geq K_U\} \mathbf{1}\{j\leq k\}.
\end{eqnarray}
We the above transition rates, we have that $(X^N(t),Y^N(t))$ is a Markov jump process.  In Figure \ref{model2_diagram}, we illustrate the transitions between states in the model proposed. Note that in order to make the illustration easier to understand, we break down $X^N(t)$ into each battery interval. Since we assume uniform arrival to scooters with battery life in all levels, we do not need to track the proportion of e-scooters in use in each battery bucket for the model to be Markovian. However, it will be needed when arrival rate depends on the battery life of e-scooters, in which case we can extend the state space to 
$$X_{u, k}^N(t)=\frac{1}{N}\sum_{i=1}^{N}\mathbf{1}\left\{\frac{k}{K}\leq B_i(t) < \frac{k+1}{K}, u_i(t)=u\right\}, \quad k=0,\cdots, K-1, u=0,1$$ 
where $u_i(t)=1$ denotes scooter $i$ is in use at time $t$ and $u_i(t)=0$ indicates when the $i^{th}$ e-scooter is idle at time $t$.

\begin{figure}[H]
\captionsetup{justification=centering}
\begin{center}
		\includegraphics[scale = 0.6]{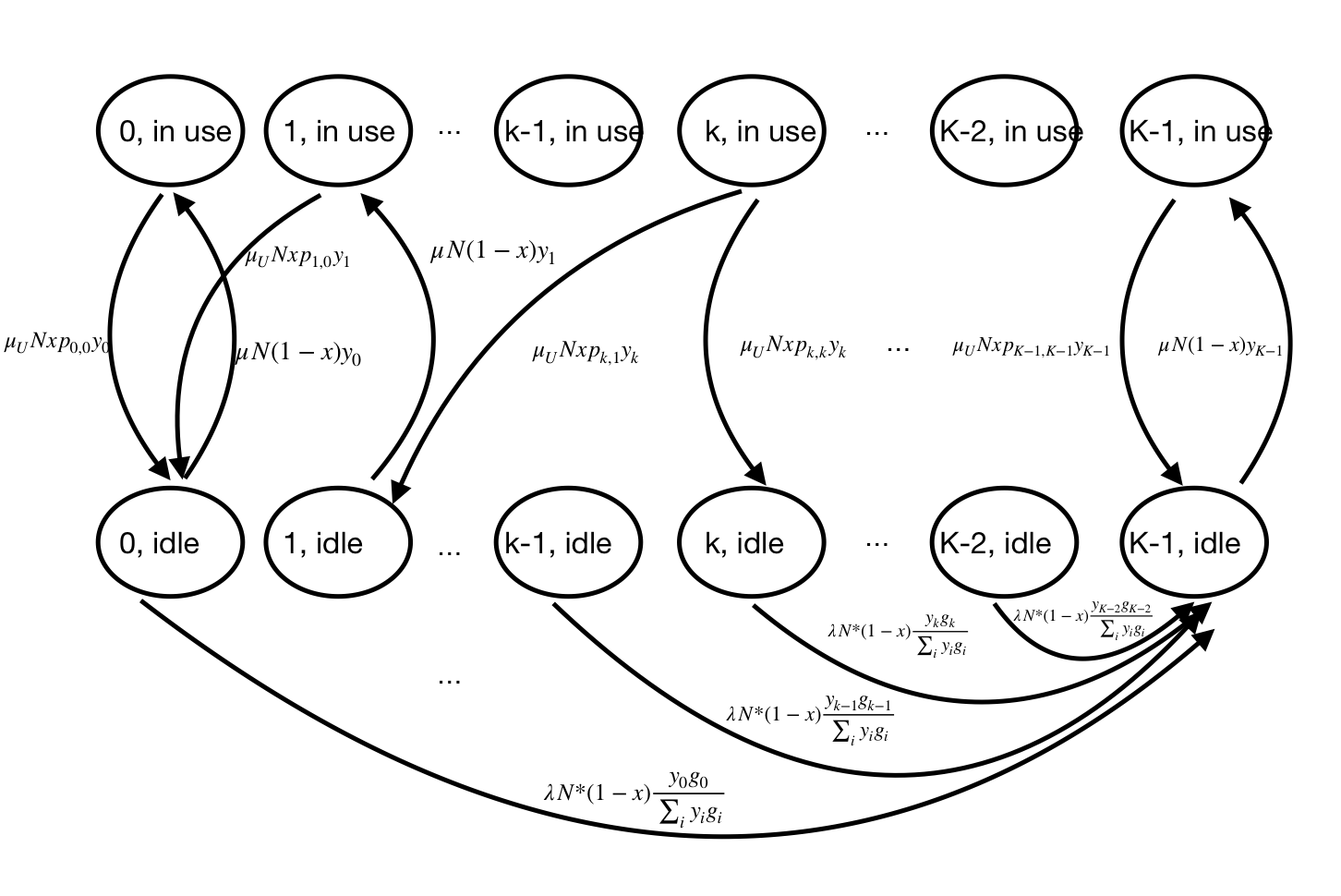}
		\end{center}
		\vspace{-.02in}
\caption{Diagram of transitions of states for the empirical process $(X^N(t),Y^N(t))$} \label{model2_diagram}
\end{figure}

Now that we have two models for the dynamic behavior of e-scooter systems, we want to understand some important behavior of the system.  Since the system is quite large and is not easy to analyze directly, we resort to using asymptotic analysis.  Thus, in the sequel we will prove mean field and central limit theorems for describe the mean and variance dynamics of the e-scooter system.  

%**************************************************************************
%**************************************************************************

\section{Mean Field Limit of Empirical Processes} \label{Mean_Field_Limit}

In this section, we prove the mean field limit for both of our empirical process of e-scooters battery life models.  A mean field limit describes the large system dynamics of the e-scooters battery life and usage over time.  Deriving the mean field limit allows us to gain insight about the average system behavior when the number of e-scooters is very large. Thus, we avoid the need to study an $N$-dimensional continuous time Markov chain and compute its steady state distribution in this high dimensional setting, which is quite intractable.  We first state the mean field limit result for the stochastic model described in Section \ref{model_1}, the empirical process of e-scooter battery life with instantaneous battery usage.

\begin{theorem}[Functional Law of Large Numbers]\label{fluid_limit}
Let $|. |$ denote the Euclidean norm in $\mathbb{R}^{K}$. Suppose that $\lim_{N\rightarrow \infty}\frac{N^*}{N}=\gamma$, and $Y^{N}(0)\xrightarrow{p} y(0)$,  then we have for $\forall \epsilon>0$
$$\lim_{N\rightarrow \infty}P\left(\sup_{t\leq t_0}|Y^N(t)-y(t)|>\epsilon\right)=0,$$
where $y(t)$ is the unique solution to the following differential equation starting at $y(0)$,
\begin{equation}\label{diff_eqn}
\shortdot{y}=f(y)
\end{equation}

where $f:[0,1]^{K}\rightarrow \mathbb{R}^{K}$ is a vector field that satisfies
\begin{eqnarray}\label{eqn:b}
f(y)&=&\sum_{k=0}^{K-1}\left[ \left(  \frac{\lambda \gamma g_k}{\sum_{i=0}^{K-1}y_ig_i} \right)(\mathbf{1}_{K-1}-\mathbf{1}_{k})+\sum_{j=0}^{k}\mu p_{kj} (\mathbf{1}_{j}-\mathbf{1}_{k})\mathbf{1}_{k\geq K_U}\right] y_k\nonumber \\
\end{eqnarray}
or componentwise for $0\leq j \leq K-1$,
\begin{eqnarray}
f_j(y)&=&\underbrace{\sum_{k=\max(K_U,j)}^{K-1}\mu p_{k,j}y_k}_{\text{battery usage from scooters in $k$-th bucket}} +\underbrace{\lambda \gamma \mathbf{1}\{j=K-1\}}_{\text{battery recharge to full}}-\underbrace{\mu y_j \mathbf{1}\{j\geq K_U\}}_{\text{battery usage from scooters in $j$-th bucket}}\nonumber\\
& &-\underbrace{\frac{\lambda \gamma y_jg_j}{\sum_{i=0}^{K-1}y_ig_i}}_{\text{battery recharge to scooters in $j$-th bucket}} .
\end{eqnarray}
\end{theorem}

\begin{proof}
The full proof is provided in the Appendix (Section \ref{Appendix}).
\end{proof}

Now we state the  the mean field limit result for the stochastic model described in Section \ref{model_2}, the empirical process of e-scooter battery life with exponentially distributed battery usage time.

\begin{theorem}[Functional Law of Large Numbers]\label{fluid_limit_model2}
Let $|. |$ denote the Euclidean norm in $\mathbb{R}^{K+1}$. Suppose that $\lim_{N\rightarrow \infty}\frac{N^*}{N}=\gamma$, and $(X^N(0),Y^{N}(0))\xrightarrow{p} (x(0),y(0))$,  then we have for $\forall \epsilon>0$
$$\lim_{N\rightarrow \infty}P\left(\sup_{t\leq t_0}|(X^N(t), Y^N(t))-(x(t),y(t))|>\epsilon\right)=0$$
where $(x(t),y(t))$ is the unique solution to the following differential equation starting at $(x(0), y(0))$
\begin{eqnarray}\label{diff_eqn}
\shortdot{x}&=&f_x(x, y),\\
\shortdot{y}&=&f_y(x, y),
\end{eqnarray}

where $f=(f_x,f_y):[0,1]\times [0,1]^{K}\rightarrow \mathbb{R}\times \mathbb{R}^{K}$ is a vector field that satisfies
\begin{eqnarray}
f_x(x,y)=\underbrace{\mu(1-x)\sum_{k=K_U}^{K-1}y_k}_{\text{customers picking up scooters}} -\underbrace{\mu_U x \sum_{k=K_U}^{K-1}y_k}_{\text{customers drop off scooters}},
\end{eqnarray}

\begin{eqnarray}
f_y(x,y)&=&\sum_{k=0}^{K-1}\left[ \left(  \frac{\lambda \gamma (1-x) g_k}{\sum_{i=0}^{K-1}y_ig_i} \right)(\mathbf{1}_{K-1}-\mathbf{1}_{k})+\sum_{j=0}^{k}\mu_U x p_{kj} (\mathbf{1}_{j}-\mathbf{1}_{k})\mathbf{1}_{k\geq K_U}\right] y_k,\nonumber \\
\end{eqnarray}

or componentwise for $0\leq j \leq K-1$,
\begin{eqnarray}
f_y(x,y)(j)&=&\underbrace{\sum_{k=\max(K_U,j)}^{K-1}\mu _U x p_{k,j}y_k}_{\text{battery usage from scooters in $k$-th bucket}} +\underbrace{\lambda \gamma(1-x) \mathbf{1}\{j=K-1\}}_{\text{battery recharge to full}}\nonumber\\
& &-\underbrace{\mu_U x y_j \mathbf{1}\{j\geq K_U\}}_{\text{battery usage from scooters in $j$-th bucket}}-\underbrace{\frac{\lambda \gamma(1-x) y_jg_j}{\sum_{i=0}^{K-1}y_ig_i}}_{\text{battery recharge to scooters in $j$-th bucket}}.
\end{eqnarray}
\end{theorem}
\begin{proof}
The proof ideas for Theorem \ref{fluid_limit_model2} follow easily from the proof of Theorem \ref{fluid_limit} so we do not prove them in this paper. 
\end{proof}

%**************************************************************************
%**************************************************************************
\iffalse
\subsection{Steady State Analysis}
By the existence and uniqueness of the fluid limit, $y(t)$ has a unique steady state $\bar{y}$ which satisfies

$$\sum_{k=\max(K_U,j)}^{K-1}\mu p_{k,j}\bar{y}_k +\lambda \gamma \mathbf{1}\{j=K-1\} =\mu \bar{y}_j \mathbf{1}\{j\geq K_U\}+\frac{\lambda \gamma \bar{y}_jg_j}{\sum_{i=0}^{K-1}\bar{y}_ig_i}$$
Specifically,
\begin{equation}
\sum_{k=K_U}^{K-1}\mu p_{k,j}\bar{y}_k  =\frac{\lambda \gamma \bar{y}_jg_j}{\sum_{i=0}^{K-1}\bar{y}_ig_i}, \quad 
0\leq j<K_U
\end{equation}
\begin{equation}
\sum_{k=j}^{K-1}\mu p_{k,j}\bar{y}_k  =\mu \bar{y}_j+\frac{\lambda \gamma \bar{y}_jg_j}{\sum_{i=0}^{K-1}\bar{y}_ig_i}, \quad K_U\leq j<K-1
\end{equation}
\begin{equation}
\mu p_{K-1,K-1}\bar{y}_{K-1} +\lambda \gamma =\mu \bar{y}_{K-1}+\frac{\lambda \gamma \bar{y}_{K-1}g_{K-1}}{\sum_{i=0}^{K-1}\bar{y}_ig_i},  \quad j=K-1
\end{equation}
\fi

Both mean field limits provide insights to the stochastic model, by providing ordinary differential equations that describe the mean proportion of scooters in a particular interval of battery life.  In the first model, the dimension is reduced from $N$ scooters to $K$ intervals where $K$ is generally much lower than $N$.  In the second model, the dimension is only increased to $K+1$, which is still much smaller than $N$, the number of e-scooters.  We will describe in the sequel how these mean field limits can be used to construct staffing algorithms for agents who will swap out the batteries when they are low.  However, the mean field limits only describe the mean dynamics of the stochastic models and say nothing about the stochastic fluctuations around the mean field limits.  In the next section, we prove central limit theorems, centering around the mean field limits for our stochastic models.  The central limit theorems will provide some rigorous support for confidence intervals around the mean field limits.  
%**************************************************************************
%**************************************************************************

\section{Central Limit Theorem of Empirical Process} \label{Central_Limit}

In this section, we derive the diffusion limit of our stochastic empirical process of scooters battery life model.  Diffusion limits are critical for obtaining a deep understanding of the sample path behavior of stochastic processes around their mean.  One reason is that diffusion limits describe the fluctuations around the mean or mean field limit and can help understand the variance or the asymptotic distribution of the stochastic process being analyzed.  We define our diffusion scaled e-scooters sharing model by subtracting the mean field limit from the scaled stochastic process and rescaling it by $\sqrt{N}$.  Thus, we obtain the following expression for the diffusion scaled scooter battery life empirical process
\begin{equation}
D^{N}(t)=\sqrt{N}(Y^{N}(t)-y(t)).
\end{equation}
Now we state the functional central limit theorem for the empirical process described in Section \ref{model_1}.
\begin{theorem}[Functional Central Limit Theorem]\label{difftheorem}
Consider $D^{N}(t)$ in $\mathbb{D}(\mathbb{R}_{+},\mathbb{R}^{K})$ with the Skorokhod $J_{1}$ topology, and suppose that 
$\limsup_{N\rightarrow \infty}\sqrt{N}\left(\frac{N^*}{N}-\gamma\right)< \infty$. 
If $D^{N}(0)$ converges in distribution to $D(0)$, then $D^{N}(t)$ converges to the unique Ornstein Uhlenbeck (OU) process solving $D(t)=D(0)+\int_{0}^{t}f'(y(s))D(s) ds+M(t)$ in distribution, where $f'(y)$ is specified as follows,
\begin{eqnarray}\label{b'}
\frac{\partial f_j(y)}{\partial y_k}&=&\mu p_{k,j}\mathbf{1}\{k\geq \max\{j,K_U\}\}-\mu \mathbf{1}\{k=j\geq K_U\}\nonumber\\
& &+\frac{\lambda \gamma g_kg_jy_j}{(\sum_{i=0}^{K-1}y_ig_i)^2}-\frac{\lambda \gamma g_j}{\sum_{i=0}^{K-1}y_ig_i}\mathbf{1}\{j= k\},
\end{eqnarray}
and $M(t) = (M_{0}(t),\cdots,M_{K-1}(t))\in \mathbb{R}^{K}$ is a real continuous centered Gaussian martingale, with Doob-Meyer brackets given by
\begin{eqnarray}
%& &
\boldlangle M_{k}(t),M_{j}(t)\boldrangle 
&=&\begin{cases}
\int_{0}^{t}\left[\sum_{i=\max(K_U,k+1)}^{K-1}\mu p_{i,k}y_i(s) +\left(\mu(1-p_{k,k})\mathbf{1}_{\{k\geq K_U\}}\right.\right.\nonumber\\
+\left.\left.\frac{\lambda \gamma g_k}{\sum_{i=0}^{K-1}y_i(s)g_i} \right)y_k(s)\right]ds, & k=j<K-1 \nonumber\\
\int_{0}^{t}\left[\mu(1-p_{K-1,K-1})y_{K-1}(s)+\lambda \gamma\left(1- \frac{g_{K-1}y_{K-1}(s)}{\sum_{i=0}^{K-1}y_i(s)g_i}\right)\right]ds, & k=j=K-1\\
-\int_{0}^{t} \left[\mu p_{k,j}y_k(s)+\frac{\lambda \gamma g_j}{\sum_{i=0}^{K-1}y_i(s)g_i}y_j(s)\mathbf{1}\{k=K-1\}\right]ds,  & j<k, k\geq K_U\nonumber\\
0. & \text{otherwise}
\end{cases}
\end{eqnarray}
Define $\mathcal{A}(t)=f'(y(t))$, $\mathcal{B}(t)=\left(\frac{d}{dt}\boldlangle M_i(t),M_j(t)\boldrangle  \right)_{ij}$, then the covariance matrix $\Sigma (t)=\mathrm{Cov}[D(t),D(t)]$ satisfies
\begin{equation}
\frac{d\Sigma(t)}{dt}=\Sigma(t)\mathcal{A}(t)^\top+\mathcal{A}(t)\Sigma(t)+\mathcal{B}(t).
\end{equation}
Moreover, componentwise,  for $i=0,1,\cdots, K-2$,
\begin{eqnarray}
\frac{d\Sigma_{ii}(t)}{dt} &=&2\sum_{k=\max\{i,K_U\}}^{K-1}\Sigma_{ik}\mu p_{k,i}-2\left( \frac{\lambda \gamma g_i}{\sum_{i=0}^{K-1}y_ig_i}+\mu \mathbf{1}\{i\geq K_U\}\right)\Sigma_{ii} +2\sum_{k=0}^{K-1}\Sigma_{ik}\frac{\lambda \gamma g_kg_iy_i}{(\sum_{i=0}^{K-1}y_ig_i)^2}\nonumber\\
& &+\sum_{k=\max(i+1,K_U)}^{K-1}\mu p_{k,i}y_k +\left(\mu(1-p_{i,i})\mathbf{1}\{i\geq K_U\}+\frac{\lambda \gamma g_i}{\sum_{i=0}^{K-1}y_ig_i} \right)y_i,
\end{eqnarray}
 and for $i=K-1$,
\begin{eqnarray}
\frac{d\Sigma_{K-1,K-1}(t)}{dt} 
&=&-2\left( \frac{\lambda \gamma g_{K-1}}{\sum_{i=0}^{K-1}y_ig_i}+\mu \right)\Sigma_{ii}+2\sum_{k=0}^{K-1}\Sigma_{ik}\frac{\lambda \gamma g_kg_{K-1}y_{K-1}}{(\sum_{i=0}^{K-1}y_ig_i)^2}+\mu(1-p_{K-1,K-1})y_{K-1}\nonumber\\
& &+\lambda \gamma\left(1- \frac{g_{K-1}y_{K-1}}{\sum_{i=0}^{K-1}y_ig_i}\right) ,
\end{eqnarray}
for $0\leq i<j\leq K-1$,
\begin{eqnarray}
\frac{d\Sigma_{ij}(t)}{dt}
&=&\sum_{k=\max\{j,K_U\}}^{K-1}\Sigma_{ik}\mu p_{k,j}-\left( \frac{\lambda \gamma g_j}{\sum_{i=0}^{K-1}y_ig_i}+\mu \mathbf{1}\{j\geq K_U\}\right)\Sigma_{ij}+\sum_{k=0}^{K-1}\Sigma_{ik}\frac{\lambda \gamma g_kg_jy_j}{(\sum_{i=0}^{K-1}y_ig_i)^2}\nonumber\\
& &\sum_{k=\max\{i,K_U\}}^{K-1}\Sigma_{jk}\mu p_{k,i}-\left( \frac{\lambda \gamma g_i}{\sum_{i=0}^{K-1}y_ig_i}+\mu \mathbf{1}\{i\geq K_U\}\right)\Sigma_{ij}+\sum_{k=0}^{K-1}\Sigma_{jk}\frac{\lambda \gamma g_kg_iy_i}{(\sum_{i=0}^{K-1}y_ig_i)^2}\nonumber\\
& &-\mu p_{j,i}y_j\mathbf{1}\{j\geq K_U\}-\frac{\lambda \gamma g_i}{\sum_{i=0}^{K-1}y_ig_i}y_i\mathbf{1}\{j=K-1\}.
\end{eqnarray}

\begin{proof}
In order to prove Theorem \ref{difftheorem}, we need to prove the following four results listed below step by step.  
  
\begin{itemize}
\item[1).](Lemma \ref{martingale_brackets})$\sqrt{N}M^{N}(t)$ is a family of martingales independent of $D^{N}(0)$ with Doob-Meyer brackets given by
\begin{eqnarray}
& &\boldlangle \sqrt{N}M^{N}_k(t),\sqrt{N}M^{N}_j(t)\boldrangle\nonumber\\
&=&\begin{cases}
\int_{0}^{t}\left[\sum_{i=\max(K_U,k+1)}^{K-1}\mu p_{i,k}Y^N_i(s) +\left(\mu(1-p_{k,k})\mathbf{1}_{\{k\geq K_U\}}\right.\right. & \nonumber\\
+\left.\left.\lambda N^*\frac{ g_k}{N\sum_{i=0}^{K-1}Y^N_i(s)g_i} \right)Y^N_k(s)\right]ds, & k=j<K-1 \nonumber\\
\int_{0}^{t}\left[\mu(1-p_{K-1,K-1})Y_{K-1}^N(s)+\lambda \frac{N^*}{N}\left(1- \frac{g_{K-1}Y_{K-1}^N(s)}{\sum_{i=0}^{K-1}Y_i^N(s)g_i}\right)\right]ds, & k=j=K-1\\
-\int_{0}^{t} \left[\mu p_{k,j}Y^N_k(s)+\lambda N^*\frac{ g_j}{N\sum_{i=0}^{K-1}Y_i^N(s)g_i}Y_j^N(s)\mathbf{1}\{k=K-1\}\right]ds,  & j<k, k\geq K_U\nonumber\\
0. & \text{otherwise}
\end{cases}
\end{eqnarray}

\item[2).](Lemma \ref{L2bound}) For any $T\geq 0$, $$\limsup_{N\rightarrow \infty}\mathbb{E}(|D^{N}(0)|^2)<\infty \Rightarrow \limsup_{N\rightarrow \infty}\mathbb{E}(\sup_{0\leq t\leq T}|D^{N}(t)|^2)<\infty. $$
\item[3).] (Lemma \ref{tightness}) If $(D^{N}(0))_{N=1}^{\infty}$ is tight then $(D^{N})_{N=1}^{\infty}$ is tight and its limit points are continuous.
\item[4).]	If $D^{N}(0)$ converges to $D(0)$ in distribution, then $D^{N}(t)$ converges to the unique OU process solving $D(t)=D(0)+\int_{0}^{t}f'(y(s))D(s) ds+M(t)$ in distribution.
\end{itemize}
We provide the proofs of Lemma \ref{martingale_brackets}, Lemma \ref{L2bound} and Lemma \ref{tightness} in the Appendix (Section \ref{Appendix}).   For step 4),  by Theorem 4.1 in Chapter 7 of \citet{Ethier2009}, it suffices to prove the following conditions hold
\begin{enumerate}
\item [a).]
\begin{equation}\label{condition_1}
\lim_{N\rightarrow \infty}\mathbb{E}\left[\sup_{t\leq T}|D^N(t)-D^N(t-)|^2\right]=0,
\end{equation}
\item  [b).]
\begin{equation}\label{condition_2}
\lim_{N\rightarrow \infty}\mathbb{E}\left[\sup_{t\leq T}\left|\int_{0}^{t}F^N(Y^N(s))ds-\int_{0}^{t-}F^N(Y^N(s))ds\right|^2\right]=0,
\end{equation}
\item  [c).] for $0\leq k,j\leq K-1$,
\begin{equation}\label{condition_3}
\lim_{N\rightarrow \infty}\mathbb{E}\left[\sup_{t\leq T}\left|\boldlangle \sqrt{N}M^{N}_k(t),\sqrt{N}M^{N}_j(t)\boldrangle-\boldlangle \sqrt{N}M^{N}_k(t-),\sqrt{N}M^{N}_j(t-)\boldrangle\right|^2\right]=0,
\end{equation}
\item [d).] for $0\leq k,j\leq K-1$,
\begin{equation}\label{condition_4}
\sup_{t\leq T}\left|\boldlangle \sqrt{N}M^{N}_k(t),\sqrt{N}M^{N}_j(t)\boldrangle-\boldlangle M_k(t),M_j(t)\boldrangle\right|\xrightarrow{p} 0, 
\end{equation}
\item  [e).]
\begin{equation}\label{condition_5}
\sup_{t\leq T}\left|\int_{0}^{t}\left\{\sqrt{N}[f(Y^N(s))-f(y(s))]-f'(y(s))D^{N}(s)\right\}ds\right|\xrightarrow{p} 0.
\end{equation}

\end{enumerate}

Condition (\ref{condition_1}) is easy to show by the fact that $D^N(t)$ has jump size of $1/\sqrt{N}$. Condition (\ref{condition_2}) follows by the fact that $F^N(y)$ is a Lipschitz  function of $y$ and that condition (\ref{condition_1}) holds. By Lemma \ref{martingale_brackets} and the fact that $Y^N(t)$ has jump size of $1/N$, it is also easy to show that condition (\ref{condition_3}) holds. For condition (\ref{condition_4}), it follows from Proposition \ref{driftbound} and Lemma \ref{martingale_brackets} (see for example proof of Equation (\ref{martingale_brackets_convergence}) for details).

Finally, to show condition (\ref{condition_5}), by Equation (\ref{b'}) we know that $f(y(t))$ is continuously differentiable with respect to $y(t)$. By the mean value theorem, for every $0\leq s\leq t$ there exists a vector $Z^{N}(s)$ in between $Y^{N}(s)$ and $y(s)$ such that
$$f(Y^{N}(s))-f(y(s))=f'(Z^{N}(s))(Y^{N}(s)-y(s)).$$
Therefore,
$$\int_{0}^{t}\left\{\sqrt{N}[f(Y^{N}(s))-f(y(s))]-f'(y(s))D^{N}(s)\right\}ds=\int_{0}^{t}[f'(Z^{N}(s))-f'(y(s))]D^{N}(s)ds.$$
We know that $$\lim_{N\rightarrow \infty}\sup_{t\leq T}|f'(Z^{N}(s))-f'(y(s))|=0\quad \text{in probability}$$
by the mean field limit convergence (Theorem \ref{fluid_limit}) and the uniform continuity of $f'$.
By applying Chebyshev's inequality we have that $D^{N}(s)$ is bounded in probability.   Then, by Lemma 5.6 in \citet{ko2018strong} we have that
$$\sup_{t\leq T}\left|\int_{0}^{t}\left\{\sqrt{N}[f(Y^N(s))-f(y(s))]-f'(y(s))D^{N}(s)\right\}ds\right|\xrightarrow{p} 0.$$
\end{proof}
\end{theorem}

Like in the mean field case, we now state the functional central limit theorem for the empirical process described in Section \ref{model_2}, where we consider battery usage time to be exponentially distributed.

\begin{theorem}[Functional Central Limit Theorem]
\label{difftheorem_model2}
Define $$D^{N}(t)=\sqrt{N}((X^N(t),Y^N(t))-(x(t),y(t))).$$
Consider $D^{N}(t)$ in $\mathbb{D}(\mathbb{R}_{+},\mathbb{R}^{K+1})$ with the Skorokhod $J_{1}$ topology, and suppose that 
$$\limsup_{N\rightarrow \infty}\sqrt{N}\left(\frac{N^*}{N}-\gamma\right)< \infty.$$ 
Then if $D^{N}(0)$ converges in distribution to $D(0)$, then $D^{N}(t)$ converges to the unique OU process solving $D(t)=D(0)+\int_{0}^{t}f'(x(s), y(s))D(s) ds+M(t)$ in distribution, where $f'(x,y)$ is specified as follows,
\begin{eqnarray}
\frac{\partial f_x(x,y)}{\partial x}&=&-(\mu+\mu_U)\sum_{k=K_U}^{K-1}y_k,\nonumber\\
\frac{\partial f_x(x,y)}{\partial y_k}&=&(\mu(1-x)-\mu_U x)\mathbf{1}\{k\geq K_U\}\nonumber\\
\frac{\partial f_y(x,y)(j)}{\partial x}&=&\sum_{k=\max(K_U,j)}^{K-1}\mu_Up_{kj}y_k-\lambda\gamma\left(\mathbf{1}\{j=K-1\}-\frac{g_jy_j}{\sum_{i=0}^{K-1}g_iy_i}\right)-\mu_U y_j \mathbf{1}\{j\geq K_U\},\nonumber\\
\frac{\partial f_y(x,y)(j)}{\partial y_k}&=&\mu_U x p_{k,j}\mathbf{1}\{k\geq \max\{j,K_U\}\}-\mu_U x \mathbf{1}\{k=j\geq K_U\},\nonumber\\
& & +\frac{\lambda \gamma (1-x) g_kg_jy_j}{(\sum_{i=0}^{K-1}y_ig_i)^2}-\frac{\lambda \gamma (1-x)g_j}{\sum_{i=0}^{K-1}y_ig_i}\mathbf{1}\{j= k\},\nonumber\\
\end{eqnarray}
and $M(t)=(M_{x}(t), M_{y,0}(t),\cdots,M_{y,K-1}(t))\in \mathbb{R}^{K+1}$ is a real continuous centered Gaussian martingale, with Doob-Meyer brackets given by
\begin{eqnarray}
\boldlangle M_{x}(t) \boldrangle & = & \int_{0}^{t}\left[(\mu(1-x(s))+\mu_U x(s))\sum_{k=K_U}^{K-1}y_k(s)\right]ds,\\
\boldlangle M_{x}(t), M_{y,j}(t)\boldrangle & = & \
-\int_{0}^{t}\left[\mu_U x(s)\sum_{k=\max\{j,K_U\}}^{K-1}p_{k,j}y_k(s)\right]ds,
\end{eqnarray}
\begin{eqnarray}
& &\boldlangle M_{y,k}(t),M_{y,j}(t)\boldrangle \nonumber\\
&=&\begin{cases}
\int_{0}^{t}\left[\sum_{i=\max(K_U,k+1)}^{K-1}\mu_U x p_{i,k}y_i(s) +\left(\mu_Ux(1-p_{k,k})\mathbf{1}_{\{k\geq K_U\}}\right.\right.\nonumber\\
+\left.\left.\frac{\lambda \gamma(1-x) g_k}{\sum_{i=0}^{K-1}y_i(s)g_i} \right)y_k(s)\right]ds, & k=j<K-1 \nonumber\\
\int_{0}^{t}\left[\mu_Ux(1-p_{K-1,K-1})y_{K-1}(s)+\lambda \gamma(1-x)\left(1- \frac{g_{K-1}y_{K-1}(s)}{\sum_{i=0}^{K-1}y_i(s)g_i}\right)\right]ds, & k=j=K-1\\
-\int_{0}^{t} \left[\mu_U x p_{k,j}y_k(s)+\frac{\lambda \gamma(1-x) g_j}{\sum_{i=0}^{K-1}y_i(s)g_i}y_j(s)\mathbf{1}\{k=K-1\}\right]ds,  & j<k, k\geq K_U\nonumber\\
0. & \text{otherwise}
\end{cases}\nonumber\\
\end{eqnarray}
Define $\mathcal{A}(t)=f'(x(t),y(t))$, $\mathcal{B}(t)=\left(\frac{d}{dt}\boldlangle M_i(t),M_j(t)\boldrangle  \right)_{ij}$, then the covariance matrix $\Sigma (t)=\mathrm{Cov}[D(t),D(t)]$ satisfies
\begin{equation}
\frac{d\Sigma(t)}{dt}=\Sigma(t)\mathcal{A}(t)^\top+\mathcal{A}(t)\Sigma(t)+\mathcal{B}(t).
\end{equation}
\begin{proof}
The proof ideas for Theorem \ref{difftheorem_model2} follow easily from the proof of Theorem \ref{difftheorem} so we do not prove them in this paper. 
\end{proof}
\end{theorem}

%********************************************************************************
%**************************************************************************

\section{Insights to Staffing Swappers}\label{Staffing}
In Sections \ref{Mean_Field_Limit} and \ref{Central_Limit}, we prove mean field and central limit theorems for the empirical process of battery life.  In this section, we show how to use these limits for providing insights for staffing the number of swappers to keep the number of e-scooters with low battery life below a pre-specified threshold.  Our limit theorems are useful because when the scale of the system $N$ is large enough, the empirical process representing battery life can be approximated by a normal distribution, which is extremely convenient from a computational perspective.  As a result, we can approximate the tail probability of the empirical measure process by the following expression
\begin{eqnarray}
P(Y_k^N(t)>x)\approx P(y_k(t)+\sqrt{\Sigma_{kk}(t)/N}\cdot Z>x)=1-\Phi\left(\frac{x-y_k(t)}{\sqrt{\Sigma_{kk}(t)/N}}\right).  
\end{eqnarray}
where $Z\sim N(0,1)$ and $\Phi$ is the cdf of standard normal distribution.

Similarly, if we want to consider several terms of the empirical process $(Y_{\sigma(i)})_{i}$ where $\sigma$ is a permutation on $\{0,1,\cdots,K-1\}$. We have
\begin{eqnarray}
P\left(\sum_{i=1}^{m}Y_{\sigma(i)}^N(t)>x\right)&\approx & P\left(\sum_{i=1}^{m}y_{\sigma(i)}(t)+\sqrt{\frac{1}{N}\left(\sum_{i=1}^{m}\Sigma_{\sigma(i)\sigma(i)}(t)+2\sum_{i<j}^{m}\Sigma_{\sigma(i)\sigma(j)}(t)\right)}\cdot Z>x\right)\nonumber\\
&=&1-\Phi\left(\frac{x-\sum_{i=1}^{m}y_{\sigma(i)}(t)}{\sqrt{\frac{1}{N}\left(\sum_{i=1}^{m}\Sigma_{\sigma(i)\sigma(i)}(t)+2\sum_{i<j}^{m}\Sigma_{\sigma(i)\sigma(j)}(t))\right)}}\right)
\end{eqnarray}

%Thus the mean field limit and diffusion limit results allow us to give approximations to important measures such as proportion of scooters with battery life lower than a threshold exceeds a certain level, and provides insights into how many swappers is sufficient to keep the system at an ideal state.

An important question the above analysis can help answer is how many swappers are needed to keep the proportion of e-scooters with low battery life lower than a threshold, i.e. $P(Y_0>x)<\epsilon$ for given $(x,\epsilon)$. Specifically, we construct the following algorithm for finding the solution numerically.

\begin{algorithm}\label{algo}
Given $x>0, \epsilon>0$, we have the following steps of finding the number of swappers needed to satisfy $P(Y_0^N>x)\leq \epsilon$.
\begin{enumerate}
\item Initialize $\gamma=1$.
\item Evaluate
\begin{equation}
f(\gamma)=\bar{y}_0(\gamma)+\sqrt{\frac{\bar{\Sigma}_{00}(\gamma)}{N}}\Phi^{-1}(1-\epsilon)-x
\end{equation}
where $\bar{y}_0(\gamma),\bar{\Sigma}_{00}(\gamma)$ are the limiting mean and variance of $Y^N_0(t)$ at equilibrium (which are computed from the mean field and diffusion limits).

If $f(\gamma)>0$, set $\gamma \leftarrow 2\gamma$ (double the value of $\gamma$) and repeat step 2 until $f(\gamma)<0$. Denote the final value of $\gamma$ as $\gamma_{\max}$.
\item Apply bisection method on interval $[0,\gamma_{\max}]$ to find the root $\gamma^*$ to $f(\gamma)$. Then $\gamma^*$ is the optimal number of swapper per e-scooter needed.
\end{enumerate}
\end{algorithm}

The main idea of the algorithm is to use the mean field and central limit theorems to construct quantiles for each interval of battery life. More specifically, we invert the quantiles to find the number of swappers to achieve the probabilistic performance given by system operator.  We will demonstrate the usefulness of this algorithm in the next section, which is devoted to numerical examples.  

%**************************************************************************
%**************************************************************************

\section{Numerical Examples and Simulation} \label{Numerics}

In this section, we use numerical examples and simulation results to provide better insights to the behavior of the e-scooters system. The examples validate our theoretical results in Sections \ref{Mean_Field_Limit} and \ref{Central_Limit} by showing how accurate the mean field and diffusion limits are for approximating the mean and variance of the empirical process. We also illustrate how to use our results for staffing the number of swappers.  The following simulation results are computed with the following parameter settings:

\begin{itemize}
\item The number of e-scooters $N=100$,
\item The number of swappers $N^*=50$,
\item Arrival rate of customers $\lambda=1$,
\item Arrival rate of swappers $\mu=1$,
\item Battery life bucket size $K=5$,
\item Battery threshold for riding $K_U=1$,
\item Battery usage probability $p_{ij}=\frac{1}{i+1}\mathbf{1}\{j\leq i\}$.
\end{itemize}

We initialize the battery life of scooters at time 0 to be uniform in each interval of battery life, i.e. $Y^N(0)=[0.2,0.2,0.2,0.2,0.2]$.
\begin{figure}[H]
\captionsetup{justification=centering}
\begin{center}
		\includegraphics[scale = 0.65]{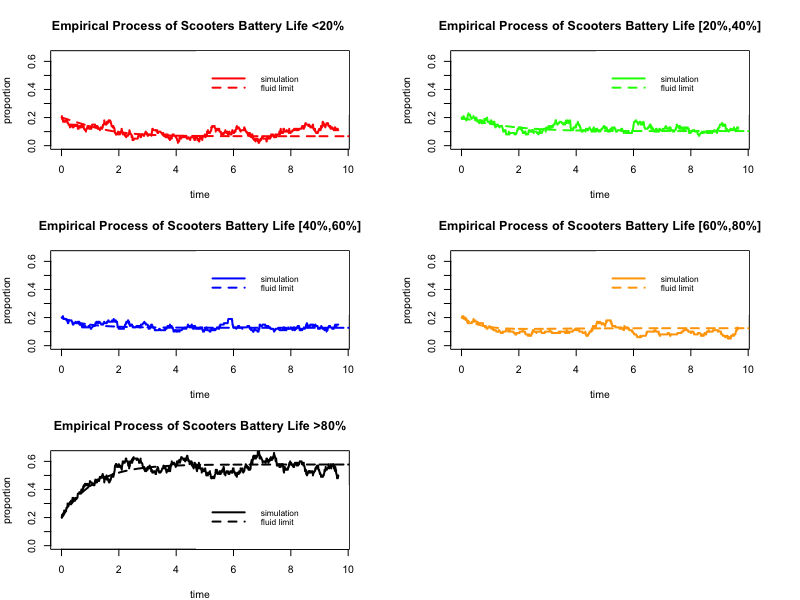}
		\end{center}
		\vspace{-.02in}
\caption{Single simulated sample path $Y^N(t)$ vs. its mean field limit $y(t)$ ($\gamma=0.5$)}  \label{Sim_1}
\end{figure}

In Figure \ref{Sim_1}, we simulate the e-scooter network according to the model dynamics where $K=5$ (intervals of battery life).  It is important to note that we have only simulated one sample path in this picture and this is not an average of sample paths.  Thus, we find that the mean field limit captures the sample path behavior of the empirical process dynamics for all of the proportions.    

\begin{figure}[H]
\captionsetup{justification=centering}
\begin{center}
		\includegraphics[scale = 0.3]{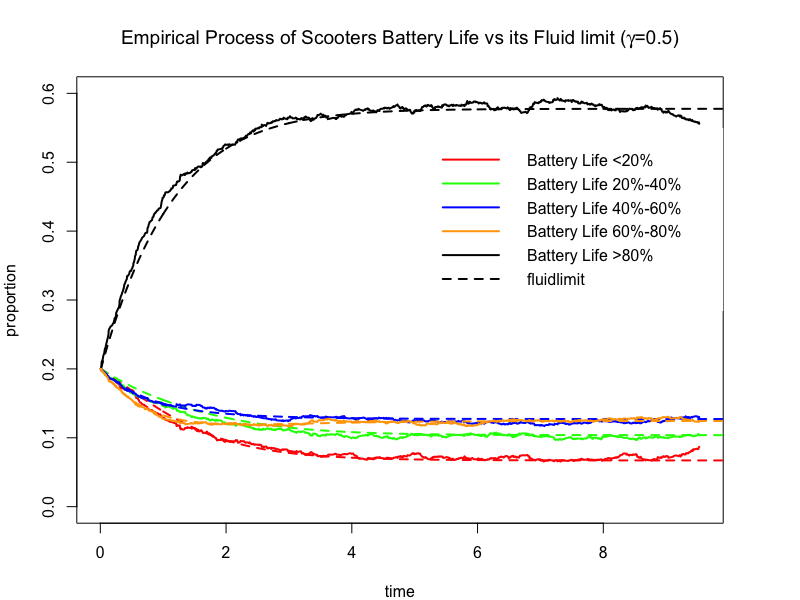}~\includegraphics[scale = 0.3]{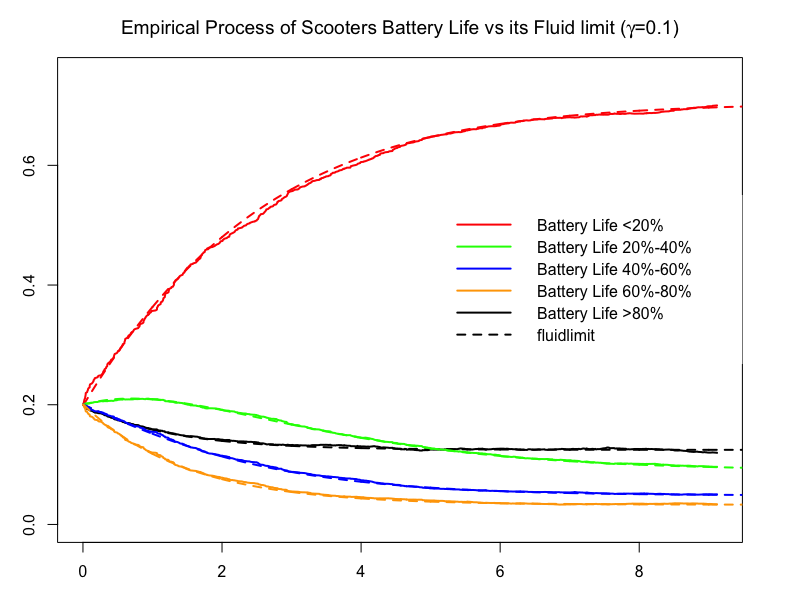}
		\end{center}
		\vspace{-.02in}
\caption{Average of 100 simulated sample paths $\mathbb{E}[Y^N(t)]$ vs. its mean field limit $y(t)$ ($\gamma=0.5$ (Left)) and ($\gamma=0.1$ (Right)) } \label{Sim_2}
\end{figure}

In Figure \ref{Sim_2}, we average the dynamics over 100 sample paths.  When compared to Figure \ref{Sim_1}, we observe that the averaged dynamics are closer to the mean field limit trajectories, which is to be expected.  We also find that the mean field limit equations capture the averaged sample path behavior of the empirical process dynamics for all of the proportions.  

\begin{figure}[H]
\captionsetup{justification=centering}
\begin{center}
		\includegraphics[scale = 0.3]{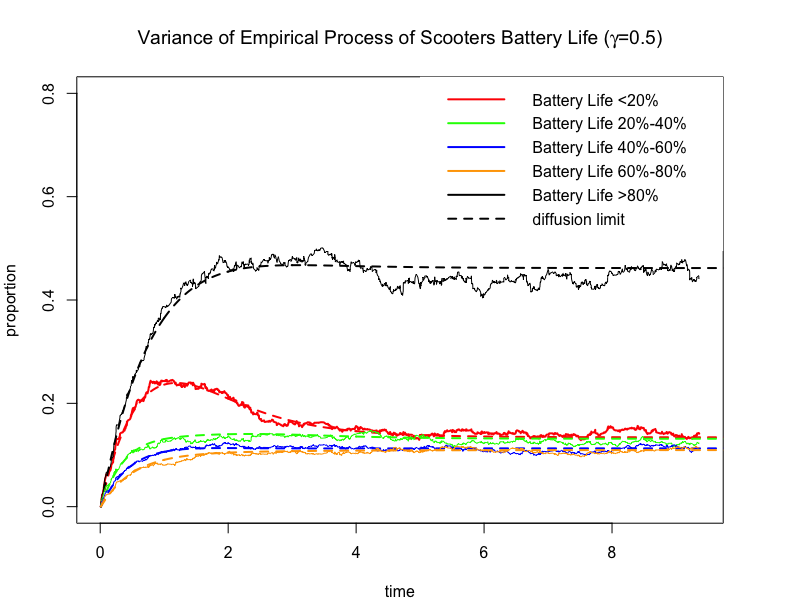}~\includegraphics[scale = 0.3]{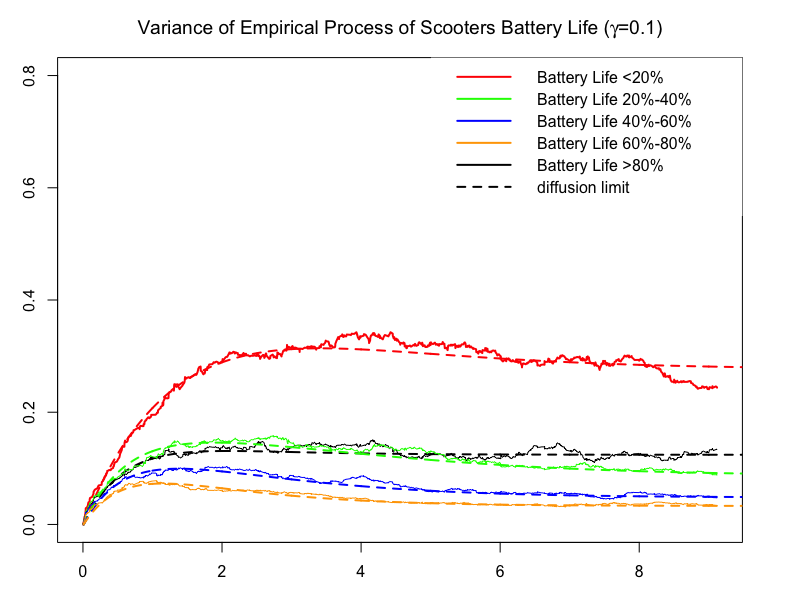}
		\end{center}
		\vspace{-.02in}
\caption{Variance of 100 simulated sample paths $\text{Var}[Y^N(t)]$ vs. its diffusion limit $\Sigma(t)$  ($\gamma=0.5$ (Left)) and ($\gamma=0.1$ (Right)) } \label{Sim_3}
\end{figure}

In Figure \ref{Sim_3}, we average the variance dynamics for each proportion over 100 sample paths.  We observe that the variance dynamics are a bit more stochastic than the mean field limit simulations.   However, we find that the variance dynamics are well approximated by the variance of the central limit theorem for the e-scooter process and for all of the proportions.  

%
%In the following plots we change the value of $\gamma$ to 0.1.

%\begin{figure}[H]
%\captionsetup{justification=centering}
%\begin{center}
%		\includegraphics[scale = 0.65]{singlepath_fluidlimit_gamma10.png}
%		\end{center}
%		\vspace{-.02in}
%\caption{Single simulated sample path $Y^N(t)$ vs. its fluid limit $y(t)$ ($\gamma=0.1$)} 
%\end{figure}

%\begin{figure}[H]
%\captionsetup{justification=centering}
%\begin{center}
%		\includegraphics[scale = 0.35]{simulation_fluidlimit_gamma10.png}
%		\end{center}
%		\vspace{-.02in}
%\caption{Average of 100 simulated sample paths $\mathbb{E}[Y^N(t)]$ vs. its fluid limit $y(t)$ ($\gamma=0.1$)} 
%\end{figure}
%
%
%\begin{figure}[H]
%\captionsetup{justification=centering}
%\begin{center}
%		\includegraphics[scale = 0.35]{var_diffusionlimit_gamma10.png}
%		\end{center}
%		\vspace{-.02in}
%\caption{Variance of 100 simulated sample paths $\text{Var}[Y^N(t)]$ vs. its diffusion limit $\Sigma(t)$  ($\gamma=0.1$)} 
%
%\end{figure}

To give an example on the implementation of Algorithm \ref{algo} in Section \ref{Staffing}, we set $(x,\epsilon)=(10\%,10\%)$, i.e. find the minimum value of $\gamma$ such that the probability that proportion of scooters with low battery life (<20\%) is more than 10\% is no greater than 10\%. The value of $\gamma$ found through the algorithm is 0.527.

In Figure \ref{hist_y0}, we validate our algorithm by running 500 simulations using the optimal number of swappers found through Algorithm \ref{algo} ($N=100$, $N^*=53$), and plot the distribution of $Y_0^N$ at equilibrium from the 500 simulations. Then we plot the 50\%, 80\%, 90\% and 95\% sample quantiles of $Y_0^N$ at equilibrium (pink solid lines) and compare them with the quantiles estimated from the algorithm (red solid lines), and we can see that the approximation using the algorithm is very close to reality. We also plot the normal distribution curve to see how good the approximation is to the distribution of $Y^N_0$.  We find that the central limit approximations capture the quantile behavior of the e-scooter system quite well. 

\begin{figure}[H]
\captionsetup{justification=centering}
\begin{center}
		\includegraphics[scale = 0.6]{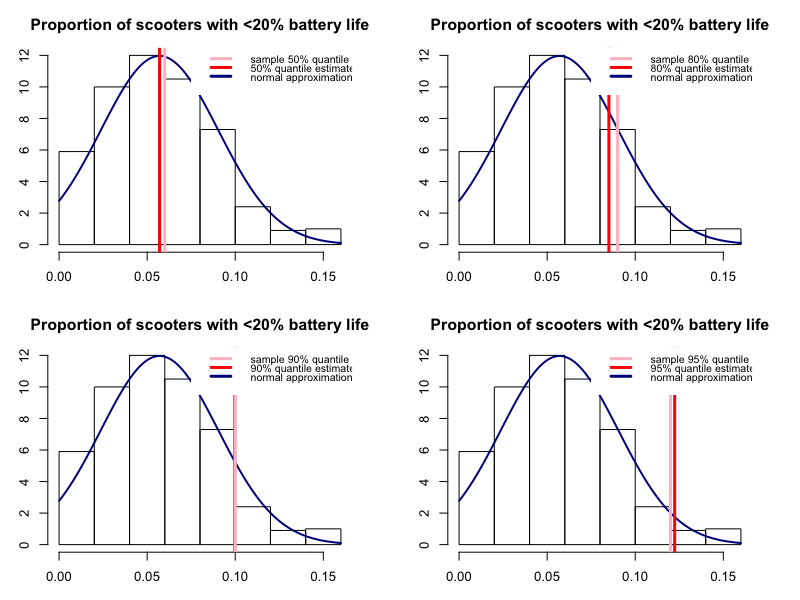}
		\end{center}
		\vspace{-.02in}
\caption{Histogram of simulated $Y_0^N$ with normal approximation} \label{hist_y0}
\end{figure}

\begin{table}[h]
\begin{center}
\caption{Optimal $\gamma$ for different values of $(x,\epsilon)$ ($\lambda=\mu=1$)} \label{table:gammas}% title of Table
\begin{tabular}{c|c|c|c|c|c|c|c} % centered columns (4 columns)
\hline\hline %inserts double horizontal lines
$(x,\epsilon)$& 0.01 & 0.05 & 0.1& 0.15& 0.2& 0.25 & 0.3\\
\hline
0.05           & 0.828 & 0.734 & 0.695& 0.672& 0.641& 0.625 & 0.609\\
0.1           & 0.598 & 0.551 & 0.527& 0.508& 0.496& 0.484 & 0.473\\
0.15          & 0.504 & 0.467 & 0.447& 0.434& 0.424& 0.414 & 0.406\\
0.2           & 0.443 & 0.412 & 0.395& 0.383& 0.375& 0.366 & 0.359\\
0.25          & 0.398 & 0.370 & 0.354& 0.344& 0.336& 0.328 & 0.322\\
0.3          & 0.361 & 0.334 & 0.320& 0.311& 0.303& 0.296 & 0.290\\
\hline %inserts single line
\end{tabular}
\end{center}
\end{table}

Table \ref{table:gammas} summarizes the optimal value of $\gamma$ (number of swappers per e-scooter, for different values of $(x,\epsilon)$. In addition, Figure \ref{gamma_surface} provides a surface plot of $\gamma$ over different values of $(x,\epsilon)$ that are given in Table \ref{table:gammas}  .

\begin{figure}[H]
\captionsetup{justification=centering}
\begin{center}
		\includegraphics[scale = 0.4]{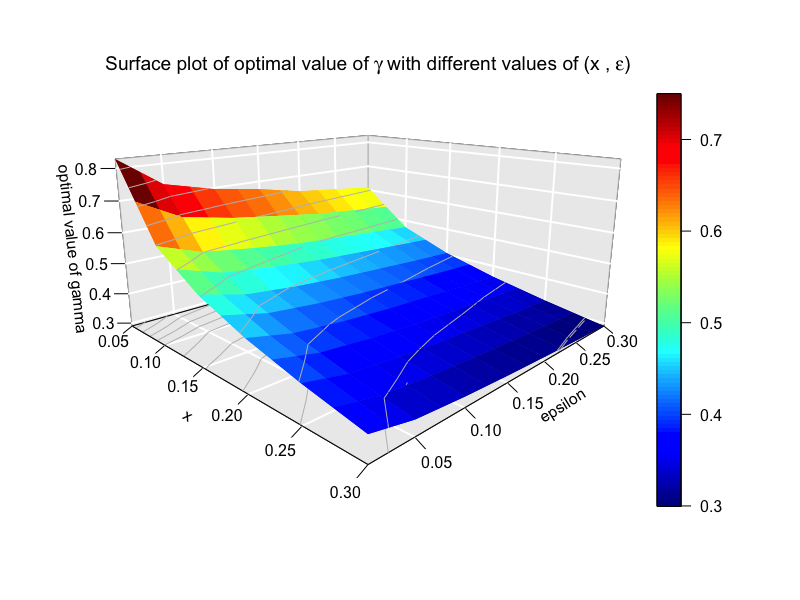}
		\end{center}
		\vspace{-.02in}
\caption{Surface plot of optimal $\gamma$ for different values of $(x,\epsilon).$} \label{gamma_surface}
\end{figure}

%**************************************************************************
%**************************************************************************

\section{Conclusion}\label{Conclusion}

In this paper, we construct two stochastic models for modeling the battery life dynamics for e-scooters in a large network.   In full generality, the model is intractable when wants to keep track of each scooter's dynamics individually because of the high dimension.  However, we propose to use empirical processes to capture the essential dynamics of battery life in e-scooter systems.   Empirical processes describe the proportion of e-scooters that have battery life in a particular interval. To this end, we prove a mean field limit and a functional central limit theorem for our e-scooter network.  We show that the mean and the variance of the empirical process can be approximated by a system of  $\frac{K^2 + 3K}{2}$ differential equations where $K$ is the number battery life intervals we want to keep. We use the mean and variance to also construct a numerical algorithm to compute the number of \textbf{swappers} needed to ensure that the fraction of scooters whose battery life is below a pre-determined threshold.  

There are many directions for future work.  As Figure \ref{Scooter_Duration_cdf} shows, the trip durations are not exponential and are closer to a lognormal distribution or gamma distribution.  An extension to general arrival and service distributions would aid in showing how the non-exponential distributions affect the dynamics of the empirical process.  Recent work by \citet{li2017nonlinear, ko2017diffusion, li2016mean, pender2017approximations} provides a Poisson process representation of Markovian arrival processes.  Thus, it might be useful to leverage this representation in future work where the arrivals and service distributions are non-renewal processes.   

Although not explicitly studied in this work, it would be interesting to explore the impact of non-stationary arrival rates and service rates.  This would undoubtedly change the underlying dynamics, however, we should mention that our analysis is easily generalizable to this setting.  Moreover, our numerical algorithm for determining the number of \textbf{swappers} does not depend on the stationarity of our model and would easily generalize to the non-stationary setting.  One important feature that is important to know in the non-stationary rate context is the size of the amplitude and the frequency of the mean field limit when the arrival rate is periodic. One way to analyze the amplitude and the frequency is to use Lindstedt's method like in \citet{novitzky2019nonlinear}.

Lastly, it is also interesting to consider a spatial model of arrivals to the e-scooter network.  In this case, we would consider customers arriving to the system via a spatial Poisson process and riders would choose among the nearest scooters with enough battery life to make their trip.  This spatial process can model the real choices that riders make and would model the real spatial dynamics of e-scooter networks.  We intend to pursue these extensions in future work.   
%**************************************************************************
%**************************************************************************

\section{Appendix}\label{Appendix}

\subsection*{Proof of Mean Field Limit Results}

\begin{proof}[Proof of Theorem \ref{fluid_limit}]
Our proof exploits Doob's inequality for martingales and Gronwall's lemma, and we use Proposition~\ref{bound}, Proposition~\ref{Lipschitz}, and Proposition~\ref{drift} in the proof, which are stated after the proof of Theorem~\ref{fluid_limit}.

Since $Y^{N}(t)$ is a semi-martingale, we have the following decomposition of $Y^N(t)$ ,
\begin{equation}\label{semiY}
Y^{N}(t)=\underbrace{Y^{N}(0)}_{\text{initial condition}}+\underbrace{M^{N}(t)}_{\text{martingale}}+\int_{0}^{t}\underbrace{F^N(Y^{N}(s))}_{\text{drift term}}ds
\end{equation}
where
$Y^{N}(0)$ is the initial condition and $M^{N}(t)$ is a family of martingales.  Moreover, $\int_{0}^{t}F^N(Y^{N}(s))ds$ is the integral of the drift term where the drift term is given by $F^N: [0,1]^{K}\rightarrow \mathbb{R}^{K}$ or
\begin{eqnarray}\label{F^N}
F^N(y)&=&\sum_{x\neq y}(x-y)Q^N(y,x)\\
&=&\sum_{k=0}^{K-1}\left[ \left( \frac{\lambda N^*}{N}   \frac{g_k}{\sum_{i=0}^{K-1}y_ig_i} \right)(\mathbf{1}_{K-1}-\mathbf{1}_{k})+\sum_{j=0}^{k}\mu p_{kj} (\mathbf{1}_{j}-\mathbf{1}_{k})\mathbf{1}_{k\geq K_U}\right] y_k\nonumber 
\end{eqnarray}

We want to compare the empirical measure $Y^N(t)$ with the mean field limit $y(t)$ defined by
\begin{equation}
y(t)=y(0)+\int_{0}^{t}f(y(s))ds.
\end{equation}

Let $|\cdot|$ denote the Euclidean norm in $\mathbb{R}^{K}$, then
\begin{eqnarray}
\left|Y^N(t)-y(t)\right|&=&\left|Y^{N}(0)+M^{N}(t)+\int_{0}^{t}F^N
(Y^{N}(s))ds-y(0)-\int_{0}^{t}f(y(s))ds\right|\nonumber \\
& =& \left|Y^{N}(0)-y(0)+M^{N}(s)+\int_{0}^{t}\left(F^N
(Y^{N}(s))-f(Y^{N}(s))\right)ds \right.\nonumber\\
& & \left.+\int_{0}^{t}(f(Y^{N}(s))-f(y(s)))ds\right|.\nonumber \\
\end{eqnarray}
Now define the random function $f^{N}(t)=\sup_{s\leq t}\left|Y^{N}(s)-y(s)\right|$, we have
\begin{eqnarray}
f^{N}(t) &\leq& |Y^{N}(0)-y(0)|+\sup_{s\leq t}|M^{N}(s)|+\int_0^t|F^N(Y^{N}(s))-f(Y^{N}(s))|ds \nonumber \\
&&+\int_{0}^{t}|f(Y^{N}(s))-f(y(s))|ds.
\end{eqnarray}
By Proposition ~\ref{Lipschitz}, $f(y)$ is Lipschitz with respect to Euclidean norm. Let $L$ be the Lipschitz constant of $f(y)$, then. 
\begin{eqnarray}
f^{N}(t)&\leq& |Y^{N}(0)-y(0)|+\sup_{s\leq t}|M^{N}(s)|+\int_0^t|F^N(Y^{N}(s))-f(Y^{N}(s))|ds \nonumber \\
&&+\int_{0}^{t}|f(Y^{N}(s))-f(y(s))|ds \nonumber \\
&\leq& |Y^{N}(0)-y(0)|+\sup_{s\leq t}|M^{N}(s)|+\int_0^t| F^N(Y^{N}(s))-f(Y^{N}(s))|ds \nonumber \\ 
&&+L\int_{0}^{t}|Y^{N}(s)-y(s)|ds \nonumber \\
&\leq& |Y^{N}(0)-y(0)|+\sup_{s\leq t}|M^{N}(s)|+\int_0^t| F^N(Y^{N}(s))-f(Y^{N}(s))|ds \nonumber \\
&&+L\int_{0}^{t}f^{N}(s)ds.
\end{eqnarray}

By Gronwall's lemma (See \citet{10.2307/1967124}), 
\begin{equation}
f^{N}(t) \leq \left(|Y^{N}(0)-y(0)|+\sup_{s\leq t}|M^{N}(s)|+\int_0^t| F^N(Y^{N}(s))-f(Y^{N}(s))|ds\right)e^{Lt}.
\end{equation}

Now  to bound $f^{N}(t)$ term by term, we define function $\alpha: [0,1]^{K}\rightarrow \mathbb{R}^{K}$ as
\begin{eqnarray}
\alpha_k(y)&=&\sum_{x\neq y}|x-y|^2Q^N(y,x)(k)\nonumber\\
&=&\begin{cases}
\frac{1}{N}\sum_{i=\max(K_U,k+1)}^{K-1}\mu p_{i,k}y_i +\frac{1}{N}\left(\mu(1-p_{k,k})\mathbf{1}_{\{k\geq K_U\}}+\frac{\lambda N^* g_k}{N\sum_{i=0}^{K-1}y_ig_i} \right)y_k, & k<K-1\nonumber\\
\frac{1}{N}\mu(1-p_{K-1,K-1})y_{K-1}+\lambda \frac{N^*}{N^2}\left(1- \frac{g_{K-1}y_{K-1}}{\sum_{i=0}^{K-1}y_ig_i}\right), & k=K-1
\end{cases}
\end{eqnarray}

and consider  the following four sets 
\begin{eqnarray}
\Omega_0 &=& \{|Y^{N}(0)-y(0)|\leq \delta \}, \\
\Omega_1 &=& \left\{\int_0^{t_0}|F^N(Y^{N}(s))-f(Y^{N}(s))|ds\leq \delta \right\}, \\
\Omega_2 &=& \left\{\int_0^{t_0}\alpha(Y^{N}(t))dt \leq A(N)t_0 \right\}, \\
\Omega_3 &=& \left\{\sup_{t\leq t_0}|M_t^{N}|\leq \delta \right\} ,
\end{eqnarray}
where $\delta=\epsilon e^{-Lt_0}/3$. Here the set $\Omega_{1}$ is to bound the initial condition, the set $\Omega_{2}$ is to bound the drift term $F^N$ and the limit of drift term $b$, and  the sets $\Omega_{2},\Omega_{3}$ are to bound the martingale $M^{N}(t)$.

Therefore on the event $\Omega_0\cap \Omega_1\cap \Omega_3$,
\begin{equation}\label{e}
f^{N}(t_0)\leq 3\delta e^{Lt_0}=\epsilon.
\end{equation}

Since $\lim_{N\rightarrow \infty}\frac{N^*}{N}=\gamma$, we can
choose large enough $N$ such that
$$\frac{N^*}{N}\leq 2\gamma.$$
Thus, we have 
\begin{eqnarray}
\alpha_k(y)&=&\sum_{x\neq y}|x-y|^2Q(y,x)(k)\nonumber\\
&\leq & \frac{1}{N}\sum_{i=\max(K_U,k+1)}^{K-1}\mu p_{i,k}y_i +\frac{1}{N}\left(\mu(1-p_{k,k})\mathbf{1}_{\{k\geq K_U\}}+2\frac{\lambda \gamma  g_k}{\sum_{i=0}^{K-1}y_ig_i} \right)y_k\nonumber \\
&\leq &\frac{1}{N}\left(\mu K +\mu +2\lambda \gamma \right)\nonumber \\
&\lesssim & O(\frac{1}{N}) .
\end{eqnarray}

Now we consider the stopping time $$T=t_0\wedge \inf \left\{t\geq 0:\int _{0}^{t}\alpha(Y^{N}(s))ds>A(N)t_0\right\},$$
 and by Proposition ~\ref{bound}, we have that
$$\mathbb{E}\left(\sup_{t\leq T}|M^{N}(t)|^2\right)\leq 4\mathbb{E}\int_{0}^{T}\alpha(Y^N(t))dt\leq 4A(N)t_0. $$
On $\Omega_2$, we have $T =t_0$, so 
$\Omega_2 \cap \Omega_3^{c}\subset \{\sup_{t\leq T}|M^{N}_t|>\delta\}$. By
Chebyshev's inequality we have that 
\begin{equation}
\mathbb{P}(\Omega_2 \cap \Omega_3^{c})\leq \mathbb{P}\left(\sup_{t\leq T}|M^{N}_t|>\delta\right)\leq \frac{\mathbb{E}\left(\sup_{t\leq T}|M^{N}(t)|^2\right)}{\delta^2}\leq 4A(N)t_0/\delta^2.
\end{equation}
Thus, by Equation (\ref{e}), we have the following result,
\begin{equation}
\begin{split}
\mathbb{P}\left(\sup_{t\leq t_0}|Y^{N}(t)-y(t)|>\epsilon\right)&\leq \mathbb{P}(\Omega_0^c\cup \Omega_1^c\cup \Omega_3^c)\\
&\leq \mathbb{P}(\Omega_2 \cap \Omega_3^{c})+\mathbb{P}(\Omega_0^{c}\cup \Omega_1^{c}\cup \Omega_2^{c})\\
&\leq 4A(N)t_0/\delta^2+\mathbb{P}(\Omega_0^{c}\cup \Omega_1^{c}\cup \Omega_2^{c})\\
&=36A(N)t_0 e^{2Lt_0}/\epsilon^2+\mathbb{P}(\Omega_0^{c}\cup \Omega_1^{c}\cup \Omega_2^{c}).
\end{split}
\end{equation}
\\
Let $A(N)=\frac{4(C+\gamma)}{N}$, then $\Omega_{2}^{c}=\emptyset$.
And since $Y_0^N\xrightarrow{p} y(0)$,   $\lim_{N\rightarrow \infty}\mathbb{P}(\Omega_{2}^{c})=0$. Therefore we have $$\lim_{N\rightarrow \infty}\mathbb{P}\left(\sup_{t\leq t_0}|Y^{N}(t)-y(t)|>\epsilon\right)=\lim_{N\rightarrow\infty}\mathbb{P}(\Omega_{1}^{c}) .$$
By Proposition \ref{drift}, $\lim_{N\rightarrow\infty}\mathbb{P}(\Omega_{1}^{c})=0$.
Thus, we proved the final result
$$\lim_{N\rightarrow \infty}\mathbb{P}\left(\sup_{t\leq t_0}|Y^{N}(t)-y(t)|>\epsilon\right)=0.$$
\end{proof}

\begin{proposition}[Bounding martingales]\label{bound}
For any stopping time $T$ such that $\mathbb{E}(T)<\infty$, we have
\begin{equation}
\mathbb{E}\left(\sup_{t\leq T}|M^{N}(t)|^2\right)\leq 4\mathbb{E}\int_{0}^{T}\alpha(Y^N(t))dt.
\end{equation}  
\begin{proof}
See proof of Proposition 4.2 (page 46) in \citet{tao2017stochastic}.	
\end{proof}
\end{proposition}

\begin{proposition}[Asymptotic Drift is Lipschitz]\label{Lipschitz}
The drift function $f(y)$ given in Equation (\ref{eqn:b}) is a Lipschitz function with respect to the Euclidean norm in $\mathbb{R}^{K}$. 
\begin{proof}
Denote $\|\cdot\|$ the Euclidean norm in $\mathbb{R}^{K}$.  Consider $y,\tilde{y}\in [0,1]^{K}$,
\begin{eqnarray}
\|f(y)-f(\tilde{y})\|&\leq& 2\left(\frac{\lambda \gamma \max_i g_i}{\min_i g_i}+\mu\right)\|y-\tilde{y}\|
\end{eqnarray}
which proves that $f(y)$ is Lipschitz with respect to Euclidean norm in $\mathbb{R}^{K}$.
\end{proof}
\end{proposition}

\begin{proposition}[Drift is Asymptotically Close to Lipschitz Drift]\label{drift}
Under the assumptions of Theorem \ref{fluid_limit}, we have for any $\epsilon>0$ and $s\geq 0$,
$$\lim_{N\rightarrow \infty}P(|F^N(Y^{N}(s))-f(Y^{N}(s))|>\epsilon)= 0.$$
\begin{proof}
\begin{eqnarray}
\left| F^N(Y^{N}(s))-f(Y^{N}(s))\right|&=&\left|\sum_{k=0}^{K-1}\left(\frac{N^*}{N}-\gamma\right)\frac{Y^{N}(s)(k)g_k}{\sum_{i=0}^{K-1}Y^{N}(s)(i)g_i}(\mathbf{1}_{K-1}-\mathbf{1}_{k})\right|\nonumber\\
&\leq & 2\left|\frac{N^*}{N}-\gamma\right|\left|\sum_{k=0}^{K-1}\frac{\max_i g_i}{\min_i g_i} \right|\nonumber\\
&=& 2\frac{K\max_i g_i}{\min_i g_i} \left|\frac{N^*}{N}-\gamma\right|  \rightarrow 0
\end{eqnarray}
\end{proof}
\end{proposition}

\subsection*{Proof of Central Limit Results}

\begin{lemma}\label{martingale_brackets}
$\sqrt{N}M^{N}(t)$ is a family of martingales independent of $D^{N}(0)$ with Doob-Meyer brackets given by
\begin{eqnarray}
& &\boldlangle \sqrt{N}M^{N}_k(t),\sqrt{N}M^{N}_j(t)\boldrangle\nonumber\\
&=&\begin{cases}
\int_{0}^{t}\left[\sum_{i=\max(K_U,k+1)}^{K-1}\mu p_{i,k}Y^N_i(s) +\left(\mu(1-p_{k,k})\mathbf{1}_{\{k\geq K_U\}}\right.\right. & \nonumber\\
+\left.\left.\lambda N^*\frac{ g_k}{N\sum_{i=0}^{K-1}Y^N_i(s)g_i} \right)Y^N_k(s)\right]ds, & k=j<K-1 \nonumber\\
\int_{0}^{t}\left[\mu(1-p_{K-1,K-1})Y_{K-1}^N(s)+\lambda \frac{N^*}{N}\left(1- \frac{g_{K-1}Y_{K-1}^N(s)}{\sum_{i=0}^{K-1}Y_i^N(s)g_i}\right)\right]ds, & k=j=K-1\\
-\int_{0}^{t} \left[\mu p_{k,j}Y^N_k(s)+\lambda N^*\frac{ g_j}{N\sum_{i=0}^{K-1}Y_i^N(s)g_i}Y_j^N(s)\mathbf{1}\{k=K-1\}\right]ds,  & j<k, k\geq K_U\nonumber\\
0. & \text{otherwise}
\end{cases}
\end{eqnarray}
\end{lemma}
%\begin{proof}
%The proof is provided in the Appendix (Section \ref{Appendix}).
%\end{proof}

\begin{proof}[Proof of Lemma \ref{martingale_brackets}]
By Dynkin's formula,
\begin{eqnarray}
\boldlangle \sqrt{N}M_k^{N}(t)\boldrangle&=&\int_{0}^{t}N\sum_{x\neq Y^{N}(s)}|x-Y^{N}(s)|^2 Q(Y^{N}(s),x)(k)ds\nonumber\\
&=&\int_{0}^{t}N\alpha_k(Y^N(s))ds\nonumber\\
&=&\begin{cases}
\int_{0}^{t}\left[\sum_{i=\max(K_U,k+1)}^{K-1}\mu p_{i,k}Y^N_i(s) +\left(\mu(1-p_{k,k})\mathbf{1}_{\{k\geq K_U\}}\right.\right.\nonumber\\
+\left.\left.\frac{\lambda N^* g_k}{N\sum_{i=0}^{K-1}Y^N_i(s)g_i} \right)Y^N_k(s)\right]ds, & k<K-1\nonumber\\
\int_{0}^{t}\left[\mu(1-p_{K-1,K-1})Y^N_{K-1}(s)+\lambda \frac{N^*}{N}\left(1- \frac{g_{K-1}Y^N_{K-1}(s)}{\sum_{i=0}^{K-1}Y^N_i(s)g_i}\right)\right]ds, & k=K-1
\end{cases}\\
&\triangleq &\int_{0}^{t}(F^N_{+}(Y^{N}(s))(k)+F^N_{-}(Y^{N}(s))(k))ds
\end{eqnarray}
where 
$$F^N_{+}(Y^{N}(s))(k)=\begin{cases}
\sum_{i=\max(K_U,k+1)}^{K-1}\mu p_{i,k}Y^N_i(s), & k<K-1 \\
\lambda \frac{N^*}{N}\left(1- \frac{g_{K-1}Y^N_{K-1}(s)}{\sum_{i=0}^{K-1}Y^N_i(s)g_i}\right), & k=K-1
\end{cases}$$
and
$$F^N_{-}(Y^{N}(s))(k)=\begin{cases}
\left(\mu(1-p_{k,k})\mathbf{1}_{\{k\geq K_U\}}+\frac{\lambda N^* g_k}{N\sum_{i=0}^{K-1}Y^N_i(s)g_i} \right)Y^N_k(s), & k<K-1 \\
\mu(1-p_{K-1,K-1})Y^N_{K-1}(s), & k=K-1
\end{cases}$$
To compute  $\boldlangle \sqrt{N}M_k^{N}(t),\sqrt{N}M_j^{N}(t)\boldrangle$ for $j<k$ and $k\geq K_U$, since
\begin{equation}
\begin{split}
&\boldlangle M_k^{N}(t)+M_j^{N}(t)\boldrangle\\
=&\int_{0}^{t}\sum_{x\neq Y^{N}(s)}\left|x_k+x_j-Y^{N}(s)(k)-Y_{j}^{N}(s)\right|^2 Q(Y^{N}(s),x)ds\\
=&\frac{1}{N}\int_{0}^{t}\left[\mu\left(\sum_{i=\max(j+1,K_U),i\neq k}^{K-1}p_{i,j}+\sum_{i=k+1}^{K-1}p_{i,k}\right)Y^N_i(s) +\left(\mu(1-p_{j,j})+\frac{\lambda \gamma g_j}{\sum_{i=0}^{K-1}Y^N_i(s)g_i} \right)Y^N_j(s) \right.\\
& \left. +\left(\mu(1-p_{k,k}-p_{k,j})+\frac{\lambda \gamma g_k}{\sum_{i=0}^{K-1}Y^N_i(s)g_i} \right)Y^N_k(s)+\lambda \gamma \{k=K-1\}\right]ds
\end{split}
\end{equation}
We have that
\begin{equation}
\begin{split}
&\boldlangle \sqrt{N}M_k^{N}(t),\sqrt{N}M_j^{N}(t)\boldrangle\\
=&\frac{N}{2}\left[\boldlangle M_k^{N}(t)+M_j^{N}(t)\boldrangle-\boldlangle M_k^{N}(t)\boldrangle-\boldlangle M_j^{N}(t)\boldrangle\right]\\
=&\frac{1}{2}\int_{0}^{t}\left[\mu\left(\sum_{i=\max(j+1,K_U),i\neq k}^{K-1}p_{i,j}+\sum_{i=k+1}^{K-1}p_{i,k}\right)Y^N_i(s) +\left(\mu(1-p_{j,j})+\frac{\lambda \gamma g_j}{\sum_{i=0}^{K-1}Y^N_i(s)g_i} \right)Y^N_j(s) \right.\\
& \left. +\left(\mu(1-p_{k,k}-p_{k,j})+\frac{\lambda \gamma g_k}{\sum_{i=0}^{K-1}Y^N_i(s)g_i} \right)Y^N_k(s)+\lambda \gamma \{k=K-1\}\right]ds\\
&-\frac{1}{2}\int_{0}^{t}(F^N_{+}(Y^{N}(s))(k)+F^N_{+}(Y^{N}(s))(j)+F^N_{-}(Y^{N}(s))(k)+F^N_{-}(Y^{N}(s))(j))ds\\
=&-\int_{0}^{t} \left[\mu p_{k,j}Y^N_k(s)+\lambda N^*\frac{ g_j}{N\sum_{i=0}^{K-1}Y_i^N(s)g_i}Y_j^N(s)\mathbf{1}\{k=K-1\}\right] ds.
\end{split}
\end{equation}
Finally, $M_k^N$ and $M_j^N$ are independent when $j,k\leq K_U$, thus in this case the Doob-Meyer brackets is equal to 0.
\end{proof}

\begin{proposition}\label{driftbound}
For any $s\geq 0$,
\begin{equation}
\limsup_{N\rightarrow \infty}\sqrt{N}\left|F^N(Y^{N}(s))-f(Y^{N}(s))\right|=0.
\end{equation}
where $F^N$ is the drift function defined in Equation (\ref{F^N}).
%\begin{proof}
%The proof is provided in the Appendix (Section \ref{Appendix}).
%\end{proof}
\end{proposition}

\begin{proof}[Proof of Proposition \ref{driftbound}]
\begin{eqnarray}
& &\limsup_{N\rightarrow \infty}\sqrt{N}\left|F^N(Y^{N}(s))-f(Y^{N}(s))\right|\nonumber\\
&=&\limsup_{N\rightarrow \infty}\sqrt{N}\left|\sum_{k=0}^{K-1}\left(\frac{N^*}{N}-\gamma\right)\frac{Y^{N}_{k}(s)g_k}{\sum_{i=0}^{K-1}Y^{N}_{i}(s)g_i}(\mathbf{1}_{K-1}-\mathbf{1}_{k})\right|\nonumber\\
&\leq & \limsup_{N\rightarrow \infty}2\sqrt{N}2\left|\frac{N^*}{N}-\gamma\right|\left|\sum_{k=0}^{K-1}\frac{\max_i g_i}{\min_i g_i} \right|\nonumber\\
&=& \limsup_{N\rightarrow \infty}2\sqrt{N}\frac{K\max_i g_i}{\min_i g_i} \left|\frac{N^*}{N}-\gamma\right| \nonumber\\
&=&0
\end{eqnarray} 
\end{proof}

\begin{lemma}[Finite Horizon Bound]\label{L2bound}
For any $T\geq 0$, if $$\limsup_{N\rightarrow \infty}\mathbb{E}\left(|D^{N}(0)|^2 \right) < \infty ,$$ then we have $$\limsup_{N\rightarrow \infty}\mathbb{E}\left(\sup_{0\leq t\leq T}|D^{N}(t)|^2 \right) < \infty .$$
\end{lemma}
%\begin{proof}
%The proof is provided in the Appendix (Section \ref{Appendix}).
%\end{proof}

\begin{proof}[Proof of Lemma \ref{L2bound}]
By Proposition \ref{driftbound}, $\sqrt{N}|F^N(Y^{N}(s))-f(Y^{N}(s))|=O(1)$, then
\begin{equation}
\begin{split}
|D^{N}(t)|&\leq |D^{N}(0)|+\sqrt{N}|M^{N}(t)|+O(1)t+\int_{0}^{t}\sqrt{N} |f(Y^{N}(s))-f(y(s))|ds\\
&\leq|D^{N}(0)|+\sqrt{N}|M^{N}(t)|+O(1)t+\int_{0}^{t}\sqrt{N}L|Y^{N}(s)-y(s)|ds\\
&=|D^{N}(0)|+\sqrt{N}|M^{N}(t)|+O(1)t+\int_{0}^{t}L|D^{N}(s)|ds.
\end{split}
\end{equation}
By Gronwall's Lemma,
$$\sup_{0\leq t\leq T}|D^{N}(t)|\leq e^{LT}\left(|D^{N}(0)|+O(1)T+\sup_{0\leq t\leq T}|\sqrt{N}M^{N}(t)|\right),$$
then
$$\limsup_{N\rightarrow \infty}\mathbb{E}\left(\sup_{0\leq t \leq T}|D^{N}(t)|^2\right)\leq e^{2LT}\left[\limsup_{N\rightarrow \infty}\mathbb{E}(|D^{N}(0)|)+O(1)T+\limsup_{N\rightarrow \infty}\mathbb{E}\left(\sup_{0\leq t\leq T}\sqrt{N}|M^{N}(t)|\right)\right]^2.$$
We know that 
$$\left[\mathbb{E}\left(\sup_{0\leq t\leq T}\sqrt{N}|M^{N}(t)|\right)\right]^2\leq N\mathbb{E}\left(\sup_{0\leq t\leq T}|M^{N}(t)|^2\right)\leq 4NA(N)T,$$
and that $A(N)=O(\frac{1}{N})$. Therefore
$$\limsup_{N\rightarrow \infty}\mathbb{E}\left(\sup_{0\leq t\leq T}\sqrt{N}|M^{N}(t)| \right)<\infty.$$
Together with our assumption that $\limsup_{N\rightarrow \infty}\mathbb{E}(|D^{N}(0)|^2)<\infty$, we have
$$\limsup_{N\rightarrow \infty}\mathbb{E}\left(\sup_{0\leq t \leq T}|D^{N}(t)|^2 \right)<\infty.$$\\
\end{proof}

\begin{lemma}\label{tightness}
If $(D^{N}(0))_{N=1}^{\infty}$ is tight then $(D^{N})_{N=1}^{\infty}$ is tight and its limit points are continuous.
\end{lemma}
\begin{proof}[Proof of Lemma \ref{tightness}]

To prove the tightness of $(D^{N})_{N=1}^{\infty}$ and the continuity of the limit points, we can apply results from \citet{billingsley2013convergence}, which implies that we only need to show the following two conditions holds for each $T>0$ and $\epsilon>0$,
\begin{itemize}
\item[(i)] 
\begin{equation}
\lim_{K\rightarrow \infty}\limsup_{N\rightarrow \infty}\mathbb{P}\left(\sup_{0\leq t\leq T}|D^{N}(t)|>K \right)=0,
\end{equation}
\item[(ii)] 
\begin{equation}
\lim_{\delta\rightarrow 0}\limsup_{N\rightarrow \infty}\mathbb{P}\left(w(D^{N},\delta,T)\geq \epsilon \right)=0
\end{equation}
\end{itemize}
where for $x\in \mathbb{D}^{d}$,
\begin{equation}
w(x,\delta,T)=\sup\left\{\sup_{u,v\in[t,t+\delta]}|x(u)-x(v)|:0\leq t\leq t+\delta\leq T\right\}.
\end{equation}
By Lemma \ref{L2bound}, there exists $C_{0}>0$ such that
\begin{eqnarray}
\lim_{K\rightarrow \infty}\limsup_{N\rightarrow \infty}\mathbb{P} \left(\sup_{0\leq t\leq T}|D^{N}(t)|>K \right) &\leq& \lim_{K\rightarrow \infty}\limsup_{N\rightarrow \infty}\frac{\mathbb{E}\left(\sup_{0\leq t\leq T}|D^{N}(t)|^2 \right)}{K^2} \\
&\leq&  \lim_{K\rightarrow \infty}\frac{C_{0}}{K^2} \\
&=&0,
\end{eqnarray}
which proves condition (i).

For condition (ii), we have that

\begin{eqnarray}
D^N(u) - D^N(v) &=& \underbrace{\sqrt{N} \cdot ( M^N(u) - M^N(v))}_{\text{first term}} + \underbrace{\int^{u}_{v} \sqrt{N} \left( F^N(Y^N(z)) - f(Y^N(z))  \right) dz}_{\text{second term}} \nonumber \\&+& \underbrace{\int^{u}_{v} \sqrt{N} \left( f(Y^N(z)) - f(y(z))  \right) dz }_{\text{third term}}
\end{eqnarray}
for any $0<t\leq u<v\leq t+\delta\leq T$.  Now it suffices to show that each of the three terms of $D^N(u) - D^N(v)$ satisfies condition (ii).  In what follows, we will show that each of the three terms satisfies condition (ii) to complete the proof of tightness. 

 For the first term, 
%\begin{eqnarray}
%\sqrt{N}|M^{N}_{u}-M^{N}_{v}|\leq |\sqrt{N}M^{N}_{u}-M_{u}|+|\sqrt{N}M^{N}_{v}-M_{v}|+|M_{u}-M_{v}|
%\end{eqnarray}
similar to the proof of Proposition ~\ref{drift}, we can show that 
$$\sup_{t\leq T}\left|F^N_{+}(Y^N(t))-f_{+}(Y^N(t))\right|\xrightarrow{p}0, \quad \sup_{t\leq T}\left|F^N_{-}(Y^N(t))-f_{-}(Y^N(t))\right|\xrightarrow{p}0.$$
And by the proof of Proposition \ref{Lipschitz}, $f_{+}(y), f_{-}(y)$ are also Lipschitz with constant $L$, then by the fact that the composition of Lipschitz functions are also Lipschitz,
\begin{eqnarray}
\max\left\{\sup_{t\leq T}|f_{+}(Y^N(t))-f_{+}(y(t))|,\sup_{t\leq T}|f_{-}(Y^N(t))-f_{-}(y(t))|\right\}\leq L\sup_{t\leq T}|Y^N(t)-y(t)|.
\end{eqnarray}
By Theorem \ref{fluid_limit}, 
\begin{equation}
\sup_{t\leq T}|Y^N(t)-y(t)|\xrightarrow{p} 0.
\end{equation}
Thus for any $\epsilon>0$,
\begin{eqnarray}
& &\lim_{N\rightarrow \infty}\mathbb{P}\left(\sup_{t\leq T}\left|\boldlangle \sqrt{N}M_k^N(t)\boldrangle- \boldlangle M_k(t)\boldrangle\right|>\epsilon\right)\nonumber \\
&=&\lim_{N\rightarrow \infty}\mathbb{P}\left(\sup_{t\leq T}\left|\int_{0}^{t}\left(F^N_{+}(Y^{N}(s))+F^N_{-}(Y^{N}(s))- f_{+}(y(s))-f_{-}(y(s))\right)ds\right|>\epsilon\right)\nonumber\\
&\leq & \lim_{N\rightarrow \infty}\mathbb{P}\left(\sup_{t\leq T}T\left|F^N_{+}(Y^N(t))- f_{+}(Y^{N}_{t})\right|>\epsilon/3\right)+\lim_{N\rightarrow \infty}\mathbb{P}\left(\sup_{t\leq T}T\left|F^N_{-}(Y^N(t))- f_{-}(Y^{N}_{t})\right|>\epsilon/3\right)\nonumber\\
& &+\lim_{N\rightarrow \infty}\mathbb{P}\left(\sup_{t\leq T}2LT\left|Y^N(t)- y(t)\right|>\epsilon/3 \right)\nonumber \\
&=& 0,
\end{eqnarray}
which implies
\begin{equation}\label{martingale_brackets_convergence}
\sup_{t\leq T}\left|\boldlangle \sqrt{N}M^N_k(t)\boldrangle- \boldlangle M_k(t)\boldrangle\right|\xrightarrow{p} 0.
\end{equation}
We also know that the jump size of $D^{N}(t)$ is $1/\sqrt{N}$, therefore
\begin{equation}
\lim_{N\rightarrow \infty}\mathbb{E}\left[\sup_{0<t\leq T}\left|M^{N}(t)-M^{N}(t-)\right| \right]=0.
\end{equation}
By Theorem 1.4 in Chapter 7 of \citet{Ethier2009},  $\sqrt{N}M^{N}(t)$ converges to the Brownian motion $M(t)$ in distribution in $\mathbb{D}(\mathbb{R}_{+},\mathbb{R}^{K+1})$. By Prohorov's theorem, $(\sqrt{N}M^{N})_{N=1}^{\infty}$ is tight. And since $M(t)$ is a Brownian motion, its sample path is almost surely continuous.

For the second term, we have by Proposition \ref{driftbound} that the quantity $ \sqrt{N} \left( F^N(Y^N(z)) - f(Y^N(z))  \right) $ is bounded for any value of $z\in [0,T]$.  Therefore, there exists some constant $C_{1}$ that does not depend on $N$ such that
\begin{equation}
\sup_{z\in [0,T]}\sqrt{N} \left| F^N(Y^N(z)) - f(Y^N(z))  \right|\leq C_{1}.
\end{equation}
Then
\begin{eqnarray}
& &\lim_{\delta\rightarrow 0}\lim_{N\rightarrow \infty}\mathbb{P}\left(\sup_{u,v\in [0,T],|u-v|\leq \delta}\int^{u}_{v} \sqrt{N} \left| F^N(Y^N(z)) - f(Y^N(z))  \right| dz > \epsilon \right)\nonumber \\ 
&\leq &  \lim_{\delta\rightarrow 0}\lim_{N\rightarrow \infty}\mathbb{P}\left(\delta \sup_{z\in [0,T]}\sqrt{N} \left| F^N(Y^N(z)) - f(Y^N(z))  \right|  > \epsilon \right)\nonumber \\
&\leq & \lim_{\delta\rightarrow 0}\mathbb{P}\left(\delta C_{1}  > \epsilon \right)\nonumber \\
&=& 0.
\end{eqnarray}
 Thus, we have proved the oscillation bound for the second term.  Finally for the third term we have that 
\begin{eqnarray}
\int^{u}_{v} \sqrt{N} \left| f(Y^N(z)) - f(y(z))  \right| dz & \leq& \int^{u}_{v} \sqrt{N} L\left| Y^N(z) - y(z) \right| dz\nonumber \\
&=& \int^{u}_{v} L \cdot \left|D^N(z)\right| dz \nonumber\\\
&\leq & L\delta \sup_{t\in [0,T]}|D^{N}(t)|.
\end{eqnarray}
By Lemma ~\ref{L2bound},  
\begin{eqnarray}
& &\lim_{\delta\rightarrow 0}\lim_{N\rightarrow \infty}\mathbb{P}\left(\sup_{u,v\in [0,T],|u-v|\leq \delta}\int^{u}_{v} \sqrt{N} \left| f(Y^N(z)) - f(y(z))  \right| dz>\epsilon\right)\nonumber\\
&\leq &\lim_{\delta\rightarrow 0}\lim_{N\rightarrow \infty}\mathbb{P}\left(L\delta\sup_{t\in [0,T]}|D^{N}(t)|>\epsilon\right)\nonumber\\
&\leq & \lim_{\delta\rightarrow 0}\lim_{N\rightarrow \infty}\frac{\mathbb{E}\left(\sup_{t\in [0,T]}|D^{N}(t)|^2\right)}{(\epsilon/L \delta)^2}\nonumber \\
&\leq &\lim_{\delta\rightarrow 0}\frac{C_{0}(L\delta)^2}{\epsilon^2}\nonumber \\
&=& 0,
\end{eqnarray}
 which implies that the oscillation bound holds for the third term.  
\end{proof}

\bibliographystyle{plainnat}
\bibliography{scooters}
\end{document}